\pdfoutput=1
\documentclass[microtype]{gtpart}
\usepackage[margin=1.45in]{geometry}
\usepackage{graphicx}
\usepackage[mathscr]{eucal}
\usepackage{amssymb}
\usepackage[dvipsnames]{xcolor}
\usepackage{enumerate, cite}
\usepackage{overpic}
\usepackage{pinlabel}

\newcommand{\br}{\mathbb{R}}

\newcommand{\bz}{\mathbb Z}
\newcommand{\bn}{\mathbb N}

\newcommand{\cE}{\mathcal E}
\newcommand{\cc}{\mathcal C}
\newcommand{\cm}{\mathscr M}
\newcommand{\cac}{\mathcal{AC}}

\newcommand{\sE}{\mathscr E}

\newcommand{\vp}{\varphi}

\newcommand{\ssm}{\smallsetminus}

\DeclareMathOperator{\mcg}{MCG}

\DeclareMathOperator{\Homeo}{Homeo}

\renewcommand{\co}{\colon\thinspace}

\newtheorem{Thm}{Theorem}[section]
\newtheorem{Thm*}{Theorem}
\newtheorem{Prop}[Thm]{Proposition}
\newtheorem{Lem}[Thm]{Lemma}
\newtheorem{Cor}[Thm]{Corollary}
\newtheorem{Cor*}[Thm*]{Corollary}

\newtheorem*{MainThm1}{Theorem~\ref{thm:main}}
\newtheorem*{MainThm2}{Theorem~\ref{thm:meager}}

\theoremstyle{definition}
\newtheorem{Def}[Thm]{Definition}

\newtheorem{Rem}[Thm]{Remark}

\numberwithin{equation}{section}

\title{Mapping class groups with the Rokhlin property}

\author{Justin Lanier}
\address{Department of Mathematics \\ University of Chicago \\ Chicago, IL 60637} 
\email{jlanier@math.uchicago.edu}

\author{Nicholas G. Vlamis}
\address{Department of mathematics \\ CUNY Graduate Center \\ New York, NY 10016, and \newline Department of Mathematics \\ CUNY Queens College \\ Flushing, NY 11367}
\email{nvlamis@gc.cuny.edu}


\begin{document}  

\begin{abstract}
We classify the connected orientable 2-manifolds whose mapping class groups have a dense conjugacy class.
We also show that the mapping class group of a connected orientable 2-manifold has a comeager conjugacy class if and only if the mapping class group is trivial.

\end{abstract}

\maketitle

%\vspace{-0.5in}

%-------------------
% Introduction
%-------------------

\section{Introduction}

A topological group has the \emph{Rokhlin property} if it contains a dense conjugacy class.
The Rokhlin property is a statement about the dynamics of a topological group acting on itself by conjugation; in fact, in the setting of Polish groups, the Rokhlin property is equivalent to this group action being topologically transitive.  
Examples of groups with the Rokhlin property include the symmetric group on a countably infinite set, the homeomorphism groups of the Cantor set and the Hilbert cube, the group of orientation-preserving homeomorphisms of an even-dimensional sphere, the automorphism groups of the random graph and the countably infinite-rank free group, and the isometry group of the rational Urysohn space.
(See the introduction of \cite{KechrisRosendal} for more history, context, and references with regards to these examples.)

The goal of this article is to classify all connected orientable 2-manifolds\footnote{A surface can have boundary, so we use the term 2-manifold when we want to stress there is no boundary.  This is further clarified in Section~\ref{sec:preliminaries}.}  whose mapping class groups have the Rokhlin property.
The mapping class group \( \mcg(S) \) of an orientable 2-manifold \( S \) is the group of homotopy classes of orientation-preserving homeomorphisms.
Viewing \( \mcg(S) \) as a quotient of the group of orientation-preserving homeomorphisms \( S \to S \), denoted \( \Homeo^+(S) \), we equip  \( \mcg(S) \) with the quotient topology with respect to the compact-open topology on \( \Homeo^+(S) \).

The statement of our main theorem, Theorem~\ref{thm:main}, relies on terminology introduced by Mann--Rafi \cite{MannRafi}, namely the notions of maximality and self-similarity for end spaces.
We will introduce the relevant definitions in detail in Section \ref{sec:preliminaries}; in the meantime, we will follow the theorem statement with several examples. 

\begin{MainThm1}
The mapping class group of a connected orientable 2-manifold has the Rokhlin property if and only if the manifold is either the 2-sphere or a non-compact manifold whose genus is either zero or infinite and whose end space is self-similar with a unique maximal end.
\end{MainThm1}

Despite the restrictive conditions in Theorem~\ref{thm:main}, there are uncountably many 2-manifolds whose mapping class groups have the Rokhlin property. Other than the plane and the sphere, whose mapping class groups are trivial, the two simplest examples are the \emph{Loch Ness monster surface}---the connected orientable one-ended infinite-genus 2-manifold---and the \emph{flute surface}---the plane with an infinite closed discrete set removed (these are depicted in Figure~\ref{fig:S1}).
The Loch Ness monster surface has a unique end, so it is maximal; the flute surface has a unique non-isolated end, which is the unique maximal end.

On the other hand, there are several ways for a mapping class group to fail to have the Rokhlin property as we will see.
One mapping class group that fails to have the Rokhlin property is the mapping class group of a 2-sphere with a Cantor set removed (the so-called \emph{Cantor tree surface}).
This is surprising for two reasons.
First, as already noted, both the group of homeomorphisms of the Cantor set and the group of orientation-preserving homeomorphisms of the 2-sphere have the Rokhlin property, and yet this property is not passed to the mapping class group.
Second, the Cantor tree surface and the Loch Ness monster surface are the standard examples of the two ways in which a surface can have a self-similar end space.
In previous work, the theory of mapping class groups of surfaces with self-similar end spaces has tended to coincide  (see \cite{MannRafi, APV2, LanierLoving}).

\begin{figure}[h]
 \labellist
 \small\hair 2pt
 \endlabellist  
\centering
\includegraphics[width=.9\textwidth]{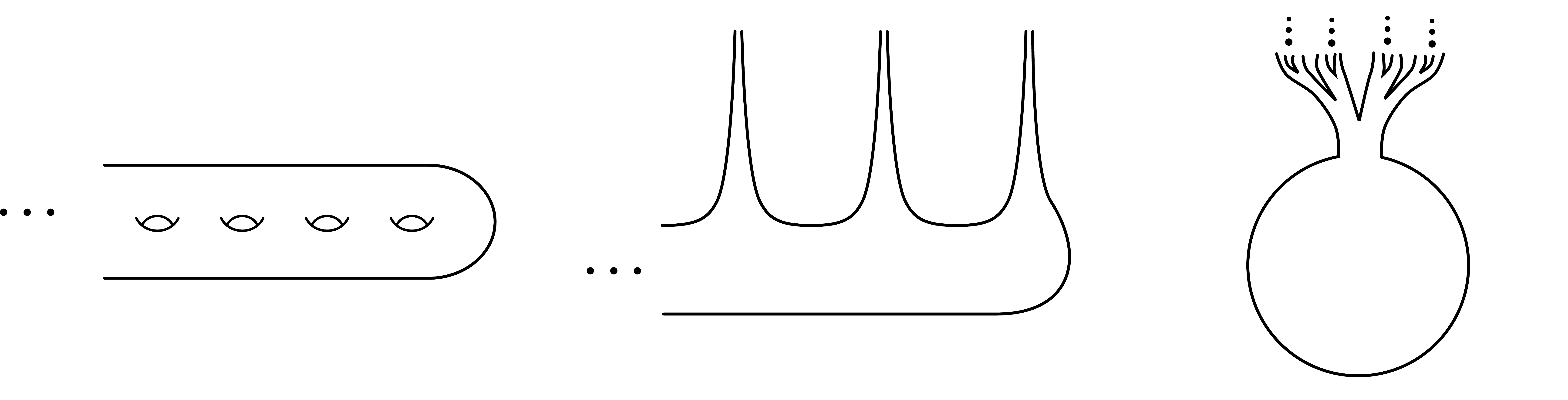}
\caption{The Loch Ness monster surface, the flute surface, and the Cantor tree surface.}
\label{fig:S1}
\end{figure}

In Section~\ref{sec:example}, we explain how to explicitly construct a dense conjugacy class when one exists; however, this construction is auxiliary to our proof of Theorem~\ref{thm:main}.
In particular, instead of directly studying conjugacy classes, we focus on the dynamics of the conjugation action of the group on itself to prove Theorem~\ref{thm:main}.
A topological group is said to have the \emph{joint embedding property}, or \emph{JEP} for short, if the action of the group on itself by conjugation is topologically transitive (see Definition~\ref{def:jep}).
It is a standard result that a Polish group has the JEP if and only if it has a dense conjugacy class (see Theorem~\ref{thm:ergodic}). 

A dense set need not take up much room, and this is the case for dense conjugacy classes in mapping class groups:
we finish the article by proving that nontrivial mapping class groups do not contain generic elements, that is, no conjugacy class is comeager. 

\begin{MainThm2}
The mapping class group of an orientable 2-manifold has a comeager conjugacy class if and only if the mapping class group  is trivial. 
\end{MainThm2}

As we discuss in detail later in the introduction, the motivation for Theorem~\ref{thm:meager} is to show that no nontrivial mapping class group has ample generics, which has connections to the automatic continuity property.

Theorem~\ref{thm:main} and Theorem~\ref{thm:meager} were recently and independently obtained by Hern\'andez, Hrusak, Morales, Randecker, Sedan, and Valdez \cite{HHMRSV}. 

It is natural to ask if there exist mapping class groups that fail to have the Rokhlin property, but have the \emph{virtual Rokhlin property}, that is,  if they have a closed finite-index subgroup with the Rokhlin property.
In a forthcoming sequel paper, the authors prove the existence of such mapping class groups and give a full classification of them. 
This further work is also motivated by the fact that Corollary~\ref{cor:cm-slender} below holds more generally for groups with the virtual Rokhlin property.

\subsection*{Applications}

\textbf{Homomorphisms to countable groups.}
The Rokhlin property, though topological and dynamical in nature, can have algebraic consequences when coupled with automatic continuity properties of classes of discrete groups.
A group \( G \) is called \emph{cm-slender} (resp. \emph{lcH-slender}) if, for any completely metrizable group (resp. locally compact Hausdorff group) \( H \), the kernel of every abstract homomorphism \( H \to G \) is open.

A group with the Rokhlin property cannot contain a proper open normal subgroup, and therefore every homomorphism from a completely metrizable group or locally compact Hausdorff group with the Rokhlin property to a cm-slender or lcH-slender group is trivial.
In particular, since mapping class group are completely metrizable, every homomorphism from a mapping class group with the Rokhlin property to a cm-slender group is trivial.
Using Theorem~\ref{thm:main}, this observation can be summarized as follows:

\begin{Cor}
\label{cor:cm-slender}
Let \( S \) be a connected orientable non-compact 2-manifold of either zero or infinite genus.
If the end space of \( S \) is self-similar with a unique maximal end, then every homomorphism from \( \mcg(S) \) to a cm-slender group is trivial. 
\end{Cor}

The class of cm-slender and lcH-slender groups is quite broad and includes free and free abelian groups \cite{Dudley}; torsion-free word-hyperbolic groups, Baumslaug--Solitar groups,  and Thompson's group \( F \) \cite{ConnerCorson};  non-exceptional spherical Artin groups (e.g.~braid groups) \cite{CorsonPreservation}; and any torsion-free subgroup of a mapping class group of an orientable finite-type surface \cite{BogopolskiAbstract}.
More generally, in \cite{ConnerCorson}, Conner--Corson give several algebraic and geometric properties that imply a group is cm- and lcH-slender as well as show that these properties are closed under direct products, free products, and graph products.
Also, note that in the sense of Gromov \cite{Gromov}, a generic finitely generated group is torsion-free hyperbolic; hence, a generic finitely generated group is cm- and lcH-slender.
Further, Conner \cite{ConnerPrivate} conjectures a countable group is cm-slender if and only if it is torsion free and does not contain an isomorphic copy of \( \mathbb Q \).

It has proven difficult to find proper normal countable-index subgroups of mapping class groups of surfaces with self-similar end spaces;
Corollary~\ref{cor:cm-slender} provides an explanation of this difficulty in the case of a unique maximal end.
Conner's conjecture would suggest that to find a normal subgroup of countably infinite index, then one would need to construct a homomorphism to the rationals: in \cite{Domat}, Domat and Dickmann do just that for the case of the mapping class group of the Loch Ness monster surface.
Their homomorphism is (necessarily) discontinuous; in contrast, in the case of the mapping class group of the Cantor tree surface, where every homomorphism to the rationals must be continuous \cite{MannAutomatic2} (see discussion below on automatic continuity), the second author \cite{VlamisThree} has shown that no normal countable-index subgroups exist.

\textbf{Infinite-degree symplectic group.}
Let \( V \) be an infinite-rank \( \bz \)-module with countable basis \( \{a_n, b_n : n \in \bn\} \)  and let \( w \) be a symplectic form such that \( w(a_n,a_m)=w(b_n,b_m) = 0 \) and \( w(a_n,b_m) = \delta_{n,m} \) for all \( n,m \in \bz \).
The \emph{infinite-degree integral symplectic group}, denoted \( \mathrm{Sp}(\bn,\bz) \), is the group of automorphisms of \( V \) preserving \( w \). 
In \cite[Corollary~3.3]{FHV}, it is shown that there is a continuous epimorphism \( \mcg(L) \to \mathrm{Sp}(\bn,\bz) \), where \( L \) is the Loch Ness monster surface (the topology on \( \mathrm{Sp}(\bn,\bz) \) is defined analogously to the presentation given in Section \ref{sec:mcg} for the mapping class group, see \cite[Section 3]{FHV} for more details).
It immediately follows from Theorem~\ref{thm:main} that \( \mathrm{Sp}(\bn,\bz) \) has the Rokhlin property:

\begin{Cor}
\label{cor:sp}
The group \( \mathrm{Sp}(\bn,\bz) \) has the Rokhlin property.
\end{Cor}

More generally, Hensel, Fanoni, and the second author in \cite{FHV} give a characterization of the image of \( \mcg(S) \) in \( \mathrm{Sp}(\bn,\bz) \) for an arbitrary connected orientable infinite-type 2-manifold \( S \) under the action of \( \mcg(S) \) on its first homology.
It follows that this image group has the Rokhlin property whenever \( S \) has a unique maximal end and its genus is either zero or infinite.

We note that it is possible to directly apply the methods in Section~\ref{sec:single} to \( \mathrm{Sp}(\bn,\bz) \) to prove Corollary~\ref{cor:sp} without reference to mapping class groups.

\textbf{Zero-dimensional spaces.}
The end space of a 2-manfiold is a compact, second-countable, zero-dimensional, Hausdorff topological space, and moreover, every such space can be realized as the end space of a 2-manifold.  
The mapping class group of a 2-manifold naturally acts on its space of ends, and the induced homomorphism from the mapping class group to the homeomorphism group of its space of ends is continuous and surjective.
We therefore have the following corollary of Theorem~\ref{thm:main}:

\begin{Cor}
\label{cor:zero-dimensional}
If a compact, second-countable, zero-dimensional, Hausdorff topological space \( E \) is self-similar with a unique maximal point, then its homeomorphism group \( \Homeo(E) \) has the Rokhlin property.  
\end{Cor}

As a specific example of Corollary~\ref{cor:zero-dimensional}, we recover the fact that the symmetric group on a countably infinite set has the Rokhlin property---it is known that this symmetric group has a comeager conjugacy class \cite{TrussGeneric}.
More generally, Corollary~\ref{cor:zero-dimensional} gives an uncountable family of examples of compact metric spaces whose homeomorphisms groups have the Rokhlin property.
Constructing such spaces was the goal of Glasner--Weiss in \cite{GW}.
The examples they produced were the Cantor Set, the Hilbert cube, and even-dimensional spheres (with the restriction to orientation-preserving homeomorphisms). 

\subsection*{Motivation}
A Polish group \( G \) has the \emph{automatic continuity property}, or \emph{ACP}, if every homomorphism from \( G \) to a separable topological group is continuous (see \cite{RosendalAutomatic} for a survey). 
Groups with the ACP exhibit a deep connection between their algebra and topology.
For example, the homeomorphism group of any closed manifold \cite{RosendalAutomatic2,MannAutomatic} have the ACP.

A group \( G \) has \emph{ample generics} if for each \( n \in \bn \) there is a comeager orbit of the diagonal conjugacy action of \( G \) on \( G^n \).
In \cite{KechrisRosendal}, Kechris--Rosendal show that if a group has ample generics, then it has the ACP.
For example, the homeomorphism group of the Cantor set has ample generics \cite{Kwiatkowska}, and hence the ACP.
Observe that in order to have ample generics, a group must have a comeager conjugacy class and hence also a dense conjugacy class. 

Unfortunately, as a consequence of Theorem~\ref{thm:meager}, we see that no mapping class group of a non-simply connected orientable 2-manifold has ample generics, which we record here.

\begin{Cor}
The mapping class group of an orientable 2-manifold has ample generics if and only if it is trivial.
\end{Cor}

Note, this result does not preclude a mapping class group from having the ACP; it only prevents the approach using ample generics.
In fact, the question of which mapping class groups have the ACP has turned out to be quite complicated:
it is known that the mapping class group of the 2-sphere minus a Cantor set has the ACP \cite{MannAutomatic2} and the mapping class group of the Loch Ness monster does not \cite{Domat}.
(Mann \cite{MannAutomatic2} gives other examples of mapping class groups with and without the ACP.)

\subsection*{Outline}

In Section~\ref{sec:preliminaries} we gather necessary preliminaries, and in particular recall the work of Mann--Rafi \cite{MannRafi}.
Their work allows us to split the proof of Theorem~\ref{thm:main} into three main cases.  
These three cases are considered in the following three sections: Section~\ref{sec:non-displaceable}, Section~\ref{sec:self-similar}, and Section~\ref{sec:doubly-pointed}.
In Section~\ref{sec:proof} we prove Theorem~\ref{thm:main}, and we finish in Section~\ref{sec:meager} by proving Theorem~\ref{thm:meager}.

\subsection*{Acknowledgments}
The authors thank Gregory Conner for sharing the definition of and articles on cm-slender groups. 
The authors thank Dan Margalit and Benjamin Weiss for comments on a draft of the article.
The authors also thank Justin Malestein and Jing Tao for pointing out an error in an earlier version of the article. 
Finally, the authors thank the referee for their careful reading of the article and their helpful comments.

The first author acknowledges support from the National Science Foundation under Grant No. DGE-1650044 and Grant No. DMS-2002187. The second author acknowledges support from PSC-CUNY Award \#63524-00 51. 
The authors also acknowledge the support of the American Mathematical Society and Simons foundation:
this project began during a visit of the first author to the second author's institution that was supported by the second author's AMS-Simons Travel Grant. 

%-------------------
% Preliminaries
%-------------------

\section{Preliminaries}
\label{sec:preliminaries}

With the goal of making this article more accessible to those studying groups from either  a topological or geometric perspective, we give a detailed preliminary section.
We first establish the equivalence between the Rokhlin property and topological mixing in the setting of Polish groups. Then we briefly recall some basic surface and mapping class group theory.  We finish with describing the binary relation on the space of ends of a surface introduced by Mann and Rafi.
For more background on mapping class groups of infinite-type surfaces, we refer the reader to the survey \cite{AramayonaVlamis}.

\subsection{Polish groups and the joint embedding property}

A topological space is \emph{Polish} if it separable and completely metrizable; a topological group is \emph{Polish} if it is Polish as a topological space.
In this subsection, we give a standard dynamical reinterpretation of the Rokhlin property in the setting of Polish groups that we will use throughout the article to establish Theorem~\ref{thm:main}.

\begin{Def}
\label{def:jep}
A topological group \( G \) has the \emph{joint embedding property}, or \emph{JEP}, if given any two nonempty open sets \( U \) and \( V \) in \( G \) there exists \( g \in G \) such that \( U \cap V^g \neq \varnothing \) (here, \( V^g = \{ gfg^{-1} \co f \in V\} \)).
This is equivalent to the conjugation action of \( G \) on itself being topologically transitive.
\end{Def}

The following theorem is a common trick to find a dense conjugacy class in a Polish group.
For convenience, we provide the proof given by Kechris and Rosendal in \cite[Theorem~2.1 and its following remark]{KechrisRosendal}.

\begin{Thm}
\label{thm:ergodic}
Let \( G \) be a Polish group.
Then, \( G \) has a dense conjugacy class if and only if \( G \) has the JEP.
\end{Thm}

\begin{proof}
First assume that \( G \) has the JEP.
Let \( \mathscr B \) be a countable basis for \( G \), and, for each \( U \in \mathscr B \), define
\[
D_U = \bigcup_{g\in G} U^g.
\]
The JEP implies that \( D_U \) is dense in \( G \) and it is clear that \( D_U \) is open.
The Baire category theorem tells us that \( G \) is a Baire space and in particular the set 
\[
D = \bigcap_{U\in \mathscr B} D_U
\]
is dense.
By construction, the conjugacy class of any element in \( D \) is dense.

For the converse, assume that the conjugacy class of \( g \) in \( G \) is dense.
Given open sets \( U_1 \) and \( U_2 \) of \( G \) there exist \( h_1 \) and \( h_2 \) in \( G \) such that \( h_igh_i^{-1} \in U_i \) for \( i \in \{1,2\} \).
It follows that \( U_1 \cap U_2^h \neq \varnothing \), where \( h = h_1h_2^{-1} \); hence, \( G \) has the JEP.
\end{proof}

\subsection{The classification of surfaces}

The standard reference for the material in this subsection is \cite{Richards}.

A \emph{2-manifold} (resp. \emph{surface}) is a second-countable Hausdorff topological space in which every point has a neighborhood homeomorphic to an open subset of the plane (resp. the closed half-plane). The \emph{boundary} of a surface \( S \), denoted \( \partial S \), is the set of non-manifold points; the \emph{interior} of \( S \) is the set of manifold points, namely \( S \ssm \partial S \).
A surface is of \emph{finite type} if it can be realized as a compact surface with a finite number of points removed from its interior. 
A 2-manifold is of \emph{infinite type} if it is not of finite type.
A surface is \emph{planar} if it is homeomorphic to a subset of \( \br^2 \); equivalently, a surface is planar if it has genus zero and it is not homeomorphic to the 2-sphere.

An \emph{exiting sequence} in a manifold \( M \) is a sequence of connected open subsets \( \{\Omega_n\}_{n\in\bn} \) such that, for every \( n \in \bn \), \( \partial \Omega_n \) is compact, \( \Omega_{n+1} \subset \Omega_n \), and \( \bigcap_{n\in\bn} \Omega_n = \varnothing \). 
Two exiting sequences \( \{\Omega_n\}_{n\in\bn} \) and \( \{\Omega'_n\}_{n\in\bn} \) are equivalent if for every \( n \in \bn \) there exists \( m \in \bn \) such that \( \Omega_m \subset \Omega'_n \) and \( \Omega'_m \subset \Omega_n \).
An \emph{end} of \( M \) is the equivalence class of an exiting sequence; let \( \sE(M) \) denote the set of all such equivalence classes.

We now explain how to topologize \( \sE(M) \).
Given a  subset \( \Omega \) of \( M \) with compact boundary, define
\[
\widehat \Omega = \{ e= [\{\Omega'_n\}_{n\in\bn}] \in \sE(M) : \text{there exists } n \in \bn \text{ such that } \Omega'_n \subset \Omega\}.
\]
The space of ends of \( M \), also denoted \( \sE(M) \), is the set of all ends of \( M \) equipped with the topology generated by the sets of form \( \widehat \Omega \) with \( \Omega \subset M \) open with compact boundary.
With this topology, the space of ends of a manifold is compact, totally disconnected, second countable, and Hausdorff; in particular, it is homeomorphic to a closed subset of the Cantor set. 

If \( \Omega \) is an open subset of \( M \) with compact boundary and \( e \) is an end  of \( M \) such that \( e \in \widehat \Omega \), then we say that \( \Omega \) is a \emph{neighborhood} of \( e \) in \( M \). 
An end of a surface is \emph{planar} if it has a neighborhood in the surface that is homeomorphic to an open subset of the plane; otherwise, it is \emph{non-planar} and, in an orientable surface, every neighborhood of the end in the surface has infinite genus.
The set of non-planar ends, denoted \( \sE_{np} \), is a closed subset of \( \sE \). 
With this setup, we can now state the classification of connected orientable surfaces with compact boundary.

\begin{Thm}
Let \( S \) and \( S' \) be connected orientable surfaces of the same (possibly infinite) genus and with (possibly empty) compact boundary with the same number of boundary components.
Then, \( S \) and \( S' \) are homeomorphic if and only if there exists a homeomorphism \( \sE(S) \to \sE(S') \) sending \( \sE_{np}(S) \)  onto \( \sE_{np}(S') \). 
\end{Thm}

As an example of this classification in action, a pair of 2-manifolds is depicted in Figure~\ref{fig:S22}. 
Since each has a single end, and  both are non-planar, the pair of 2-manifolds are in fact homeomorphic.

\begin{figure}[h]
 \labellist
 \small\hair 2pt

 \endlabellist  
\centering
\includegraphics[width=.8\textwidth]{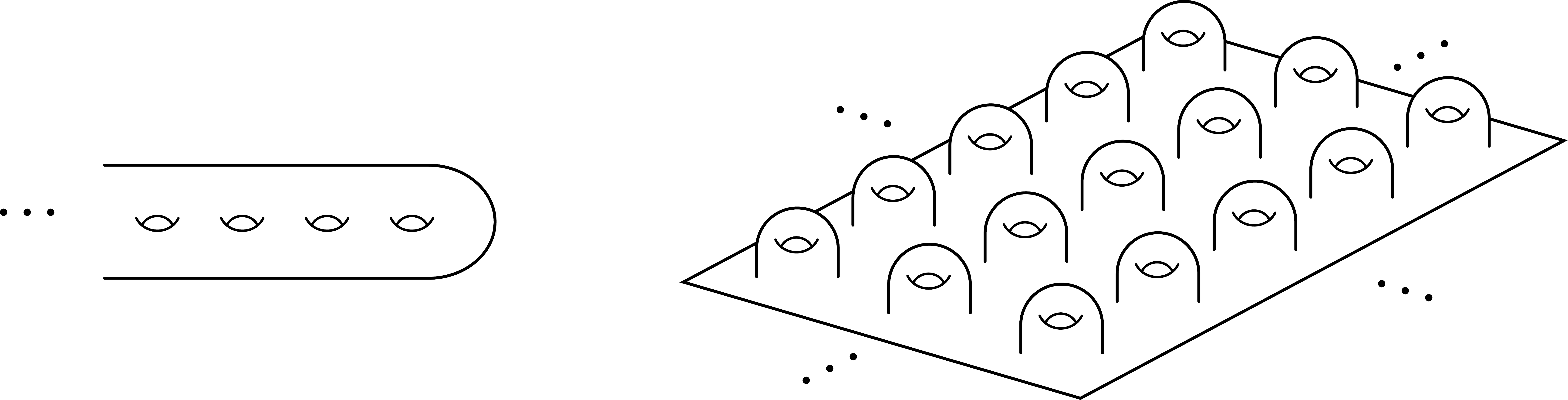}
\caption{Two realizations of the Loch Ness monster surface.}
\label{fig:S22}
\end{figure}

\subsection{Mapping class groups and curves on surfaces}
\label{sec:mcg}

We now introduce mapping class groups and facts about their topology (see the recent survey \cite{AramayonaVlamis} for further details regarding topology of mapping class groups).
We also introduce curve graphs, subsurface projections, and Alexander systems.

\subsubsection{Mapping class groups and their topology}
The \emph{mapping class group}, denoted \( \mcg(S) \), of an orientable surface \( S \) is the group of homotopy classes of orientation-preserving homeomorphisms \( S \to S \) with the additional restriction that every homeomorphism and homotopy fix the boundary of \( S \) pointwise. 
We denote the group of orientation-preserving homeomorphisms \( S \to S \) fixing \( \partial S \) pointwise by \( \Homeo_\partial^+(S) \); when \( \partial S = \varnothing \) (or if we want to drop the assumption on \( \partial S \)), we simply write \( \Homeo^+(S) \). 
We equip \( \Homeo_\partial^+(S) \) with the compact-open topology and equip \( \mcg(S) \) with the corresponding quotient topology. 
With this topology, the mapping class group is Polish.
Note that the mapping class group of a surface with compact boundary is discrete if and only if the surface is of finite type.

It will be useful to have a more combinatorial description of the topology of \( \mcg(S) \).
A simple closed curve on a surface is \emph{essential} if  no component of its complement is homeomorphic to a disk, once-punctured disk, or annulus.
Let \( \cc(S) \) denote the set of isotopy classes of essential simple closed curves on \( S \).
Given a subset \( A \) of \( \cc(S) \), we define the subset \( U_A \subset \mcg(S) \) by
\[
U_A = \{ \vp \in \mcg(S) : \vp(a) = a \text{ for all } a \in A \}.
\]
The elements of \( \{U_A : A \subset \cc(S) \text{ and } |A|< \infty\} \) together with their \( \mcg(S) \)-translates form a basis for the topology on \( \mcg(S) \) defined above whenever \( S \) has empty boundary and \( \mcg(S) \) has trivial center (in the case of boundary, it is possible to slightly modify the above definition, see \cite[Section 2.4]{APV}).
Note that there are only finitely many connected orientable 2-manifolds whose mapping class groups have nontrivial center; in particular, the center of the mapping class group of an infinite-type 2-manifold is always trivial \cite[Proposition~2]{LanierCenters}, and there are only finitely many finite-type 2-manifolds whose mapping class group has nontrivial center \cite[Theorem~5.6]{ParisGeometric}. 

\subsubsection{Curve graphs}
All of our proofs will rely on studying the intersection of essential curves. 
Given two elements \( a,b \in \cc(S) \),   their \emph{geometric intersection number} is the quantity \[ i(a,b) = \min\{ |\alpha \cap \beta| : \alpha \in a, \beta \in b \}, \]
where \( \alpha \) and \( \beta \) are representatives of \( a \) and \( b \).
Two simple closed curves are in \emph{minimal position} if they minimize the geometric intersection number of their homotopy classes.
We say that two subsurfaces \( \Sigma_1 \) and \( \Sigma_2 \) of \( S \) have nontrivial geometric intersection if there exist \( c_1,c_2 \in \cc(S) \) such that \( c_i \) has a representative contained in \( \Sigma_i \) and \( i(c_1,c_2) \neq 0 \); equivalently, \( \Sigma_1 \) and \( \Sigma_2 \) have trivial geometric intersection if \( \Sigma_1 \) is homotopic to a subsurface disjoint from \( \Sigma_2 \).  

It will be helpful for us to consider a standard metric on the set \( \cc(S) \) and recall a few properties.
We define a symmetric binary relation on \( \cc(S) \)  relating two distinct elements of \( \cc(S) \) if they have disjoint representatives in \( S \): the set \( \cc(S) \) with this relation is known as the \emph{curve graph} of \( S \). 
In turns out that if the graph \( \cc(S) \) has at least one edge, then it is connected (see \cite[Lemma 2.1]{MasurGeometryI}), which allows us to define the metric \( d_{\cc(S)} \) to be the path metric on \( \cc(S) \) associated with the above relation and where an edge is given length~1.
For instance, if \( S \) has at least five ends, five boundary components, or genus greater than one, then \( \cc(S) \) is connected.
(In the literature, it is  standard to modify the definition of the graph in the other cases so \( \cc(S) \) is connected, but this will not be relevant here.)
In this metric, \( \cc(S) \) is infinite diameter; this is implied by the following result of Masur--Minsky \cite{MasurGeometryI} that we will use throughout the paper: if \( S \) is a connected orientable finite-type surface such that \( \cc(S) \) is connected,  then given \( D > 0 \) there exists \( f \in \mcg(S) \) such that \( d_{\cc(S)}(c,f(c)) > D \) for all \( c \in \cc(S) \); for instance, \( f \) can be taken to be a sufficiently high power of any pseudo-Anosov homeomorphism, see \cite[Proposition~4.6]{MasurGeometryI}.

\subsubsection{Subsurface projections}
Let \( \Sigma \) be a subsurface of a surface \( S \) such that the embedding of \( \Sigma \) into \( S \) induces an embedding of \( \cc(\Sigma) \) into \( \cc(S) \); for instance, if each component of \( \partial \Sigma \) is an essential separating simple closed curve in \( S \). 
Let \( c \in \cc(S) \) be such that there exists \( a \in \cc(\Sigma)  \) satisfying \( i(a,c) \neq 0 \). 
Assuming that \( \partial \Sigma \) and \( c \) are in minimal position, a \emph{projection of \( c \) to \( \cc(\Sigma) \)} is defined to be any \( b \in \cc(\Sigma) \) obtained by taking a component of the boundary of a regular neighborhood of \( \alpha \cup \partial \Sigma \), where \( \alpha \) is a component of \( c \cap \Sigma \); such a projection \( b \) always exists.
Note that if \( c \in \cc(\Sigma) \), then \( b = c \). 
Also note that if \( b' \) is another projection of \( c \) to \( \cc(\Sigma) \), then \( d_{\cc(\Sigma)}(b,b') \leq 2 \) \cite[Lemma 2.2]{MasurGeometryII}.

\subsubsection{Alexander systems}
We end this subsection by introducing the notion of a stable Alexander system (see \cite[Section~2.3]{Primer} for an introduction to Alexander systems).
A subset \( A \) of \( \cc(S) \) is a \emph{stable Alexander system} if \( U_A \) is the center of \( \mcg(S) \) (in particular, with finitely many exceptions, if \( S \) has empty boundary, then \( U_A \) is the identity). 
All finite-type surfaces have finite stable Alexander systems---examples are constructed in the proof of \cite[Theorem~5.6]{ParisGeometric}.  
We note that if \( A \) is a stable Alexander system for \( S \) and \( c \in \cc(S) \ssm A \), then \( d_{\cc(S)}(A, c) > 1 \).
Though not directly used in this article, the existence of stable Alexander systems for infinite-type surfaces  \cite{HernandezAlexander} is fundamental to a number of the results mentioned so far.

\subsection{Ordering ends and a partition of surfaces}

In \cite{MannRafi}, Mann--Rafi introduce a binary relation on the space of ends of a surface that will be crucial to our proofs.
The content of this subsection is from \cite[Section 4]{MannRafi} and we refer the reader there for further details.
We will recall the key definitions here and an important consequence.

A homeomorphism \( f \in \Homeo^+(S) \) induces a homeomorphism in \( \Homeo(\sE(S)) \), which we denote \( \widehat f \). 
If \( f \) and \( g \) are homotopic, then \( \widehat f = \widehat g \); hence, we get a homomorphism \( \mcg(S) \to \Homeo(\sE(S), \sE_{np}(S)) \), where the codomain is the group of homeomorphisms of \( \sE(S) \) fixing \( \sE_{np}(S) \) setwise. 
In fact, this homomorphism is surjective.
Moreover, given a local homeomorphism \( \rho \co \mathscr U \to \mathscr V \) between two clopen subsets of \( \sE(S) \), \( \rho \) can be extended to a global homeomorphism \( \bar \rho \co \cE(S) \to \cE(S) \) by requiring \( \bar\rho|_{\mathscr V} = \rho^{-1} \). 
This allows us to freely move between discussing local and global homeomorphisms of the end space of a surface and homeomorphisms of the underlying surface.

Let \( \preceq \) denote the binary relation on \( \cE(S) \) given by \( y \preceq x \) if, for every open neighborhood \( \mathscr U \) of \( x \), there exists an open neighborhood \( \mathscr V \) of \( y \) and \( f \in \Homeo^+(S) \) such that \( \widehat f(\mathscr V) \subset \mathscr U \).
An end \( \mu \) is \emph{maximal} if, for any other end \( e \), \( \mu \preceq e \) implies \( e \preceq \mu \). 
The set of maximal ends of \( S \) is denoted by \( \mathscr M(S) \), or simply \( \mathscr M \).
For every end \( e \), there exists \( \mu \in \mathscr M \) such that \( e \preceq \mu \). 
We say two ends, \( e_1 \) and \( e_2 \), are \emph{comparable} if  \( e_1 \preceq e_2 \) or \( e_2 \preceq e_1 \). 

A subsurface \( \Sigma \) of \( S \) is \emph{displaceable} if there exists \( f \in \Homeo^+(S) \) such that \( f(\Sigma) \cap \Sigma = \varnothing \).
The space of ends of \( \sE(S) \) is \emph{self-similar} if given a decomposition \( \sE(S) = \sE_1 \sqcup \cdots \sqcup \sE_n \) into a finite disjoint union of clopen subsets,  there exists \( i \in \{1, \ldots, n\} \) such that \( \sE_i \) contains an open homeomorpic copy of \( \sE(S) \). 
The space of ends of \( \sE(S) \) is \emph{doubly pointed} if it has exactly two maximal ends. 
Putting together various results from \cite{MannRafi}, as described in \cite[Theorem 6.1]{APV2} and its proof, we have the following theorem:

\begin{Thm}
\label{thm:partition}
If \( S \) is an connected orientable  2-manifold in which every compact subsurface is displaceable, then either
\begin{enumerate}
\item \( \mathscr M \) is a singleton,
\item \( \mathscr M \) contains exactly two points, or 
\item \( \mathscr M \) is a Cantor space in which every maximal end is comparable.
\end{enumerate}
Moreover, \( \sE(S) \) is doubly pointed in the second case and self-similar in the others.
\end{Thm}

We will use Theorem~\ref{thm:partition} to break the proof of Theorem~\ref{thm:main} into three main cases. We dedicate a section to each, and then prove the main theorem in Section~\ref{sec:proof}.

\subsection{Notational conventions}

Before continuing, we introduce some general conventions in our notation.
We will regularly be working with three classes of topological spaces: surfaces, their end spaces, and their mapping class groups.
We generally use \( S \) to denote a surface and capital Greek letters for subsets of that surface---generally, \( \Sigma \) and \( \Omega \). 
We will use capital script Roman letters---generally, \( \mathscr U, \mathscr V, \) and \( \mathscr W \)---for subsets of end spaces  of surfaces.
We will use capital Roman letters---generally \( U, V \), and \( W \)---for subsets of groups. 

For the sake of simplifying notation, we will routinely abuse notation and conflate a mapping class with a representative homeomorphism and, similarly, conflate an isotopy class of a simple closed curve with a representative on the surface.

%---------------
% Compact non-displaceable subsurface
%---------------

\section{Compact non-displaceable subsurfaces}
\label{sec:non-displaceable}

A subsurface \( \Sigma \) of a surface \( S \) is \emph{non-displaceable} if \( f(\Sigma) \cap \Sigma \neq \varnothing \) for every \( f \in \Homeo^+(S) \). 
An important example to note is that every surface with positive, finite genus contains a compact non-displaceable subsurface, namely any compact subsurface of the same genus.
It is this fact, together with Theorem~\ref{thm:nondisplaceable}, that makes the hypothesis on genus in Theorem~\ref{thm:main} necessary.
Similarly, other than the plane and the open annulus, a planar 2-manifold with finitely many isolated ends contains a compact non-displaceable surface, namely any compact subsurface such that  each complementary component of the subsurface is a neighborhood of at most one isolated end.
In addition to these examples, there are  other sources of non-displaceable subsurfaces, see \cite[Section 2]{MannRafi}.

\begin{Thm}
\label{thm:nondisplaceable}
The mapping class group of a connected orientable 2-manifold containing a  proper compact non-displaceable subsurface does not have the Rokhlin property.
\end{Thm}

\begin{proof}
Let \( S \) be an orientable surface and let \( \Sigma \) be a proper compact non-displaceable subsurface of \( S \). 
Note that under these hypotheses, \( S \) is neither the plane nor the sphere; hence, if \( S \) is of finite type, then \( \mcg(S) \) is discrete and nontrivial, in which case it cannot have a dense conjugacy class. 
We can therefore assume that \( S \) is of infinite type.
By possibly enlarging \( \Sigma \), we may assume that \( \Sigma \) is connected, that \( \Sigma \) admits a stable Alexander system, and that \( \partial \Sigma \) has at least five components, each of which is essential and separating.
(In order to guarantee these properties, \( \Sigma \) may no longer be compact, but it will still be of finite type.
The requirement on the number of boundary components is to guarantee that \( \mathcal C(\Sigma) \) is connected.)
Under these additional assumptions, the inclusion \( \Sigma \hookrightarrow S \) induces a monomorphism \( \mcg(\Sigma) \to \mcg(S) \).
We will show that there exist open sets \( U, V \subset \mcg(S)\) such that \(U \cap V^g = \varnothing\) for all \( g \in \mcg(S) \); in particular,  \( \mcg(S) \) does not have have the JEP, and hence, by Theorem~\ref{thm:ergodic}, does not have the Rokhlin property. 

Let \( A \subset \cc(S) \) be a stable Alexander system for \( \Sigma \).
Let \( f \) be a homeomorphism on the subsurface \( \Sigma \) such that \(d_{\cc(\Sigma)}(c,f(c)) > 2 \) for all \( c \in \mathcal{C}(\Sigma) \). 
Let \( U=U_A \) and let \( V=fU_A \). 
Let \( g \) be an arbitrary element of \( \mcg(S) \) and let \( v \) be an arbitrary element of \( V \). 
We need to show that \( gvg^{-1} \not\in U \). 
Note that \( v \) has a representative that maps \( \Sigma \) to itself, so we can assume \(  g  v   g^{-1} \) maps \( g(\Sigma) \) to itself.

By the non-displaceability of \( \Sigma \) and the hypothesis on \( A \), there exists \( a \in A \) such that \( a \) and \( g(\Sigma) \) have nontrivial geometric intersection; let \( b \) be a projection of \( a \) to \( \cc(g(\Sigma)) \).
Now \( g^{-1}(b) \) is an element of \( \mathcal{C}(\Sigma) \), and since \( v \in V \), we must have that in \( \Sigma \), 
\[ d_{\cc(\Sigma)}(g^{-1}(b), vg^{-1}(b)) =  d_{\cc(\Sigma)}(g^{-1}(b),fg^{-1}(b)) > 2. \]
Applying \( g \), we see that in \( g(\Sigma) \), \( d_{g(\Sigma)}(b,gvg^{-1}(b)) > 2 \). 
Now suppose that \( gvg^{-1} \in U \), so that \(  gvg^{-1}(a) =a \).  
Then, it must be that \( gvg^{-1}(b) \) is a projection of \( a \) to \( g(\Sigma) \), in which case \( d_{g(\Sigma)}(b,gvg^{-1}(b))\leq 2 \); but, this is a contradiction, and therefore \( gvg^{-1} \notin U \).
Both \( v \) and \( g \) were arbitrary, so we conclude that \( U \cap V^g = \varnothing \)  for all \( g \in \mcg(S) \).
\end{proof}

\begin{Rem}
The mapping class groups of the closed disk and the once-punctured disk are trivial, and hence have the Rokhlin property.
Other than these two surfaces, 
Theorem~\ref{thm:nondisplaceable} also holds for connected orientable surfaces with at least one compact boundary component: this follows from the observation that every surface with a compact boundary component has a non-displaceable compact subsurface. 
\end{Rem}

\section{Self-similar end space}
\label{sec:self-similar}

In this section we will prove:

\begin{Thm}
\label{thm:self-similar}
Let \( S \) be a connected orientable non-compact 2-manifold with self-similar end space and in which every compact subsurface is displaceable.
The mapping class group of \( S \) has the Rokhlin property if and only if the set of maximal ends is a singleton. 
\end{Thm}

By Theorem~\ref{thm:partition}, there are two cases to consider in the proof of Theorem~\ref{thm:self-similar}: (1) the set of maximal ends is a singleton and (2) the set of maximal ends is a Cantor set.
Moreover, as already noted, every surface of positive finite genus has a compact non-displaceable subsurface, so we may assume that all surfaces in this section have either zero or infinite genus. 
Below, we break the two cases into two subsections, Section~\ref{sec:single} and Section~\ref{sec:cantor}, respectively, showing that in the first case mapping class groups have the JEP and that they do not in the second.
In Section~\ref{sec:example}, we give an explicit example of a mapping class whose conjugacy class is dense. 
We finish with the proof of Theorem~\ref{thm:self-similar} in Section~\ref{sec:proof4}.

%-----------------------
% Single maximal point
%-----------------------

\subsection{Unique maximal end}
\label{sec:single}

For the entirety of the subsection, \( S \) will denote a connected orientable 2-manifold satisfying: (1) \( S \) is either planar or infinite genus and (2) the end space \( \sE \) of \( S \) is self-similar with a unique maximal end, call it \( \mu \).

Given a separating simple closed curve \( c \) in \( S \), let \( \Omega_c \) denote the component of \( S\ssm c \) such that \( \mu \in \widehat{\Omega}_c \) and let \( \Sigma_c = S \ssm \Omega_c \). 
Let \( G \) denote the subgroup of \( \mcg(S) \) consisting of elements with a representative that restricts to the identity on \( \Omega_c \) for some separating curve \( c \). 
Observe that the set \[ \{\widehat{\Omega}_c : c \text{ a separating simple closed curve}\} \] is a neighborhood basis for \( \mu \) and therefore \( G \) is a subgroup.
If \( S \) is the Loch Ness monster surface, then \( G \) consists of the mapping classes which have a representative that is the identity outside of a compact set: in this case, it was shown in \cite[Theorem~4]{PatelVlamis} that \( G \) is dense in the mapping class group.
This example is the motivation for the following proposition.

\begin{Prop}
\label{prop:top-generate}
Let \( S \) and \( G \) be as above. 
\begin{enumerate}
\item
The subgroup \( G \) is dense in \( \mcg(S) \).
\item
Given any separating simple closed curve \( c \) in \( S \), there exists \( h \in \Homeo^+(S) \) such that \( h(\Sigma_c) \subset \Omega_c \). 
\end{enumerate}
\end{Prop}

\begin{proof}
If \( S \) is homeomorphic to the plane, then the statement is trivial.
We will now assume that \( S \) is not homeomorphic to the plane, in which case \( S \) is of infinite type.
Let \( A \) be a finite subset of \( \cc(S) \) and let \( f \in \mcg(S) \).
We must show that there exists \( g \in G \cap fU_A \). 

Let \( c \) be a separating simple closed curve such that each curve in \( A  \) has a representative contained in \( \Sigma_c \). 
Now choose a separating simple closed curve \( b \) such that \( \Sigma_c \cup f(\Sigma_c) \) is contained in \( \Sigma_b \). 
By the self-similarity of \( \sE \),  \( \widehat \Omega_b \) contains an open subset homeomorphic to \( \sE \), and so the classification of surfaces, together with the assumption on the genus of \( S \), guarantees the existence of a separating simple closed curve \( c' \) in \( \Omega_b \) so that \( \Sigma_{c'} \) is homeomorphic to and disjoint from \(  \Sigma_c \).
It is therefore possible to choose a separating simple closed curve \( b' \) co-bounding a pair of pants\footnote{A \emph{pair of pants} is a surface homeomorphic to the 2-sphere with three pairwise-disjoint open disks removed.} with \( c \) and \( c' \). 
We can now apply the classification of surfaces to find \( g_1 \in \Homeo(\Sigma_{b'}) \) such that \( g_1(\Sigma_c) = \Sigma_{c'} \).
Note that this establishes (2) and that  \( g_1 \in G \).

Now, since \( \Sigma_{c'} \) is disjoint from \( f(\Sigma_c) = \Sigma_{f(c)} \), we can use a similar argument to find \( g_2 \in G \) such that \( (g_2\circ g_1)(\Sigma_c) = f(\Sigma_c) \). 
As \( f^{-1}\circ g_2 \circ g_1 \) fixes \( \Sigma_c \) setwise, we can  write \[ f^{-1}\cdot g_2 \cdot g_1  =  \vp \cdot g_3  , \] where  \( g_3 \in \mcg(\Sigma_c) \) and \( \vp \in \mcg(S) \) has a representative that restricts to the identity on \( \Sigma_c \). 
Since \( \vp \in U_A \), we know
\[ f^{-1} \cdot (g_2\cdot g_1 \cdot g_3^{-1}) \in U_A, \] and, since \( g_1, g_2, g_3 \in G \), we conclude that \[ g=g_2\cdot g_1 \cdot g_3^{-1} \in G \cap fU_A \] as desired.  
\end{proof}

\begin{Lem}
\label{lem:singleton-yes}
Let \( S \) be as above. Then, \( \mcg(S) \) has the JEP.
\end{Lem}

\begin{proof}
Given open subsets \( U_1 \) and \( U_2 \) of \( \mcg(S) \) we need to find \( g \in \mcg(S) \) such that \( U_1 \cap U_2^g \neq \varnothing \). 
Without loss of generality, we may assume that \( U_1 \) and \( U_2 \) are basis elements of \( \mcg(S) \), that is, there exists finite subsets \( A_1 \) and \( A_2 \) of \( \cc(S) \) and \( h_1, h_2 \in \mcg(S) \) such that \( U_i = h_i U_{A_i} \) for \( i \in \{1,2\} \). 
Note that if \( h \in f U_A \), then \( fU_A = hU_A \); hence, by Proposition~\ref{prop:top-generate}, we can assume that both \( h_1 \) and \( h_2 \) are in \( G \).

We can choose a  separating simple closed curve \( b \) so that \( h_1 \) and \( h_2 \) have representatives that restrict to the identity on \( \Omega_b \) and such that each curve in \( A_1 \cup A_2 \) has a representative in \( \Sigma_b \). 
Then, by Proposotion~\ref{prop:top-generate}(2), there exists \( g \in \mcg(S) \) such that \( g(\Sigma_b) \subset \Omega_{b} \).
Note that \( U_2^g = (h_2U_{A_2})^g = (gh_2g^{-1})U_{g(A_2)} \). 
The element \( h_1 \) has a representative that restricts to the identity on \( \Omega_b \); the element \( gh_2g^{-1} \) has a representative that restricts to the identity on \( \Omega_{g(b)} \) and hence on \( \Sigma_b \).
Therefore,  \( h_1 \) and \( gh_2g^{-1} \) commute, \( h_1 \in U_{g(A_2)} \), and \( gh_2g^{-1} \in U_{A_1} \).

Let \( h = h_1(gh_2g^{-1}) \).
Then, for every \( a \in A_1 \), since \( gh_2g^{-1} \in U_{A_1} \),
\[
h(a) = h_1\cdot(gh_2g^{-1})(a) = h_1(a),
\]
which implies \( h \in U_1 \). 
Similarly, for every \( a \in g(A_2) \),
\[
h(a) = h_1\cdot(gh_2g^{-1})(a) = (gh_2g^{-1})\cdot h_1(a) = gh_2g^{-1}(a),
\]
since \( h_1 \in U_{g(A_2)} \), which implies \( h \in U_2^g \).
We have shown that  \( h \in U_1 \cap U_2^g \); in particular, \( U_1 \cap U_2^g \neq \varnothing \)  and thus \( \mcg(S) \) has the JEP.
\end{proof}

%-----------------------
% An explicit example
%-----------------------

\subsection{Explicit examples of dense conjugacy classes}
\label{sec:example}

In this subsection, we construct a explicit mapping classes whose conjugacy class is dense.
Let \( L \) denote the Loch Ness monster surface and let \( G \)  be as in Section~\ref{sec:single}.
Then in this case, as already noted,  \( G \) consists of mapping classes that have a representative homeomorphism that restricts to the identity outside of a compact set.  

Let \( R_g \) denote the connected orientable compact surface of genus \( g \) with a single boundary component. 
Let \[ \mathcal G  = \{ (R_g, \vp) : g \in \mathbb N, \vp \in \mcg(R_g) \} .\]
The set \( \mathcal G \) is countable, so let us choose an enumeration \( \mathcal G = \{ (R_n, \vp_n) \}_{n\in\bn} \). 
Let \( \Sigma \) denote the surface obtained from \( \br^2 \) by removing the open disk centered at \( (i,j) \) of radius \( \frac14 \) for each \( (i,j) \in \bz^2 \). 
Let \( \{ \partial_n \}_{n\in\bn} \) be an enumeration of the boundary components of \( \Sigma \).
We can now construct \( L \) by taking the topological disjoint union \( \Sigma \sqcup \bigsqcup_{n\in\bn} R_n \) and forming the quotient space obtained by identifying the boundary of \( R_n \) with \( \partial_n \) via an orientation-reversing homeomorphism. This is depicted in Figure~\ref{fig:S42}.
The embedding \( R_n \hookrightarrow L \) yields a monomorphism \( \mcg(R_n) \to \mcg(L) \).
Therefore, for each \( n \in \bn \), we can view \( \vp_n \in \mcg(L) \) and define
\[
\Phi = \prod_{n\in\bn} \vp_n.
\]

\begin{figure}[h]
 \labellist
 \small\hair 2pt

 \endlabellist  
\centering
\includegraphics[width=.6\textwidth]{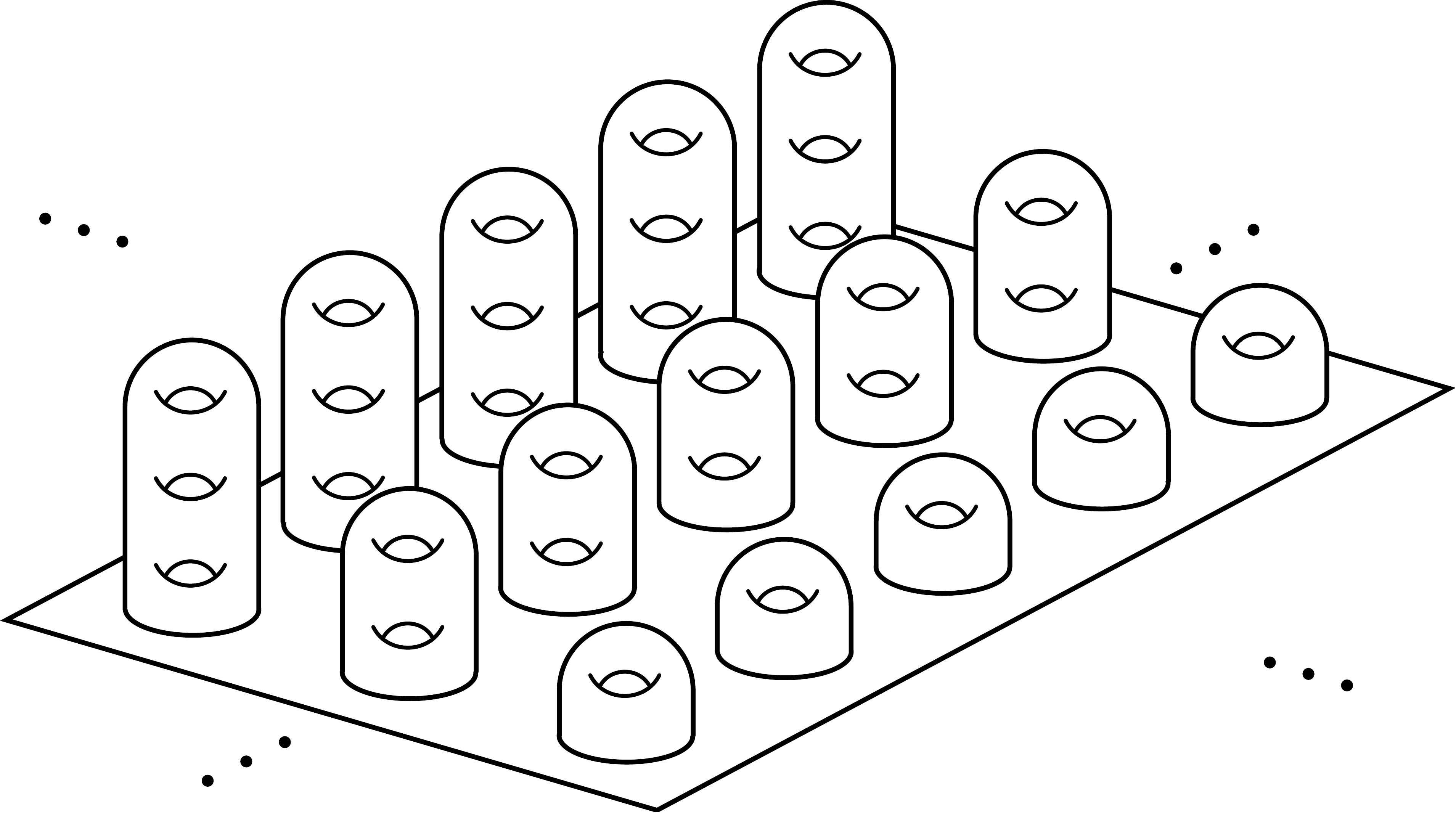}
\caption{A realization of the Loch Ness monster surface \(L\) produced by identifying \(\Sigma\) and the countably many compact surfaces \(R_n\) along pairs of boundary components.}
\label{fig:S42}
\end{figure}

\begin{Prop}
\label{prop:explicit}
The conjugacy class of \( \Phi \) is dense in \( \mcg(L) \).
\end{Prop}

\begin{proof}
Let \( A \) be a finite subset of \( \cc(S) \) and let \( f \in \mcg(L) \).
We need to show that there is a conjugate of \( \Phi \) in \( fU_A \).
By the density of \( G \), there exists \( h \in G \) such that \( h \in fU_A \), and hence \( fU_A = hU_A \).
By the definition of \( G \), there exists a compact surface \( \Omega \) of \( L \) such that \( h \) has a representative restricting to the identity outside of \( \Omega \).
Letting \( A' \) be a stable Alexander system for \( \Omega \), we have that  \( hU_{A'} \subset fU_A \).

By the classification of surfaces, for every \( g \in \bn \) and any pair of  embeddings \( \iota_1 \) and \( \iota_2 \) of \( R_g \) into \( L \), there exists a homeomorphism \( \sigma \co L \to L \) such that \( \iota_2 = \sigma \circ \iota_1 \).  
In particular, there exists a conjugate of \( \Phi \) that agrees with \( h \) when restricted to \( \Sigma \), and hence a conjugate of \( \Phi \) is contained in \( hU_{A'}\subset fU_A \). 
\end{proof}

This argument can easily be adapted to the case of the flute surface since the corresponding subgroup \( G \) from Proposition~\ref{prop:top-generate} is also countable.  
More generally, since mapping class groups are second countable, the group \( G \) from Proposition~\ref{prop:top-generate} must contain a countable dense subgroup.
Using this countable dense subgroup, the construction given above can be adapted to give an explicit construction of a dense conjugacy class for the mapping class groups of surfaces treated in Section~\ref{sec:single}.

\begin{Rem}
Note that the conjugacy class of \( \Phi \) is not the only conjugacy class that is dense in \( \mcg(L) \).
For instance, we can modify the construction as follows to obtain a distinct dense conjugacy class: when enumerating the boundary components of \( \Sigma \), leave out the components centered at \( (i,0) \) for each \( i \in \bz \).
We then attach a copy of \( R_1 \) to each of these remaining boundary components and introduce a handle shift \( h \) sending the boundary component of \( \Sigma \) with center \( (i,0) \) to the component centered at \( (i+1,0) \) for all \( i \in \bz \).  
(See the survey \cite{AramayonaVlamis} for an introduction to handle shifts.)
This handle shift \( h \)  will commute with \( \Phi \), and we define the product \( \Phi' = h\Phi \).
Then, the conjugacy class of \( \Phi' \) is dense in \( \mcg(L) \) for the same reason that the conjugacy class of \( \Phi \) is.
We can see that these two elements are not conjugate because for \( \Phi' \) there exists a simple closed curve \( c \) such that for every compact set \( K \) of \( L \) there exists \( n \in \bn \) satisfying \( (\Phi')^n(c) \cap K = \varnothing \), which is not the case for \( \Phi \). 
If a comeager conjugacy class exists in Polish group, then it is unique, and so the conjugacy classes of \( \Phi \) and \( \Phi' \) cannot both be comeager.
In fact, we will see in Section~\ref{sec:meager} that there is no comeager conjugacy class in \( \mcg(L) \) (nor in any nontrivial mapping class group).
\end{Rem}

%-----------------------
% Cantor space of maximal points
%-----------------------

\subsection{Cantor space of maximal points}
\label{sec:cantor}

Throughout this subsection, \( S \) will denote a connected orientable surface satisfying:  (1) every compact subsurface of \( S \) is displaceable (and hence \( S \) is either planar or infinite genus), and (2) the end space \( \sE \) of \( S \) is self-similar and its set of maximal ends \( \mathscr M \) is a Cantor space.
Note that, under these assumptions,  Theorem~\ref{thm:partition} implies any two maximal ends of \( S \) are comparable.

To prepare for the proof of Lemma~\ref{lem:cantor-no}, we first need to prove a lemma regarding the structure of the end space of the surfaces being considered.

\begin{Prop}
\label{prop:Exn}

For any positive integer \( m \),  if \( M \) is a discrete space with \( m \) points, then \( \sE \times M \) is homeomorphic to \( \sE \).
\end{Prop}

The main ingredient in the proof will be \cite[Lemma 4.18]{MannRafi}.
We state a version adapted to our particular scenario, which uses \cite[Remark~4.15]{MannRafi}:

\begin{Lem}[Mann--Rafi]
\label{lem:stable}
Let \( \mu \in \cm \). 
For every \( y \in \sE \), there exists an open neighborhood \( \mathscr U \) of \( y \) such that, for every open neighborhood \( \mathscr U' \) of \( y \) contained in \( \mathscr U \), the space \( \sE \sqcup \mathscr U' \) is homeomorphic to \( \sE \).
\end{Lem}

\begin{proof}[Proof of Proposition~\ref{prop:Exn}]
It is enough to show the statement holds for \( m = 2 \). 
By the assumptions on \( \sE \), there exists disjoint open subsets \( \mathscr V_1 \) and \( \mathscr V_2 \)  of \( \sE \) each of which is homeomorphic to \( \sE \).
If \( \sE = \mathscr V_1 \cup \mathscr V_2 \), then we are done; otherwise, let \( \mathscr V_3 = \sE \ssm (\mathscr V_1 \cup \mathscr V_2) \) and, for each \( y \in \mathscr V_3 \), let \( \mathscr U_y \) be the open neighborhood given by Lemma~\ref{lem:stable}; by possibly shrinking \( \mathscr U_y \) and taking the intersection with \( \mathscr V_3 \), we may assume that \( \mathscr U_y \) is clopen and \( \mathscr U_y \subset \mathscr V_3 \).
Since \( \mathscr V_3 \) is compact, there exists \( y_1, \ldots, y_m \) such that \( \mathscr V_3 = \mathscr U_{y_1} \cup\cdots\cup \mathscr U_{y_m} \).
Now, let \( \mathscr U_1' = \mathscr U_{y_1} \), \( \mathscr U_2' = \mathscr U_{y_2} \ssm \mathscr U_1' \), \( \mathscr U_3' = \mathscr U_{y_3} \ssm (\mathscr U_1' \cup \mathscr U_2') \), etc.
Then, each \( \mathscr U_i' \) is open, \( \mathscr V_3 = \bigsqcup_{i=1}^m \mathscr U_i' \), and, for each \( i \), \( \mathscr V_1 \cup \mathscr U_i' \) is homeomorphic to \( \sE \).
Therefore, \( \mathscr V_4 = \mathscr V_1 \cup \mathscr V_3 \) is homeomorphic to \( \sE \). 
Now, \( \mathscr V_2 \) and \( \mathscr V_4 \) are both open and homeomorphic to \( \sE \), they are disjoint, and their union is \( \sE \); hence, \( \sE \cong \sE \times \{1,2\} \).   
\end{proof}

\begin{Lem}
\label{lem:cantor-no}
Let \( S \) be as above. Then, \( \mcg(S) \) does not have the JEP.
\end{Lem}

\begin{proof}
Using Proposition~\ref{prop:Exn}, we can choose a compact surface \( \Sigma \subset S \) satisfying:
\begin{enumerate}
\item \( \Sigma \) is a planar surface,
\item \( \Sigma \) has five boundary components, labelled \( b_1, b_2, b_3, b_4 \) and \( b_5 \),
\item each component of \( \partial \Sigma \) is separating, and
\item each of the five corresponding components \( \Omega_1, \Omega_2, \Omega_3, \Omega_4, \Omega_5  \) of \( S \ssm \Sigma \) satisfies \( \widehat \Omega_i \) is homeomorphic to \( \sE \).
\end{enumerate}

The conditions on \( \Sigma \) guarantee that the inclusion \( \Sigma \hookrightarrow S \) induces a monomorphism \( \mcg(\Sigma) \to \mcg(S) \). 
Let \( f \in \mcg(\Sigma) \) be a homeomorphism such that \( f(b_i) = b_{i+1} \), where the indices are read modulo 5, and such that \( d_{\cc(\Sigma)}(f(c), c) > 2 \) for all \( c \in \mathcal{C}(\Sigma) \).
The existence of \( f \) is guaranteed by the results and discussion in \cite[Section~14.1.5]{Primer} and \cite[Proposition~4.6]{MasurGeometryI}.

Choose a finite stable Alexander system  \( A \) for \( \Sigma \), let \(  U_1 = fU_A \) and let \( U_2 = U_{\{b_1\}} \).
Note that every element of \( U_1 \) acts on \( \mathcal{C}(\Sigma) \) and this action agrees with that of \( f \).
We claim that \( U_1 \cap U_2^g = \varnothing \) for all \( g \in \mcg(S) \).
Let \( g \in \mcg(S) \), then we consider two cases: either \( g(b_1) \) has a representative disjoint from \( \Sigma \) or not. 
First assume that \( g(b_1) \) has a representative disjoint from \( \Sigma \), inclusive of the possibility that the curve \( g(b_1) \) is isotopic to a component of \( \partial \Sigma \).
In this case, up to isotopy, there exists a component \( \Omega_i \) of \( S \ssm \Sigma \) containing \( g(b_1) \).
But, if \( \vp \in U_2^g \), then \( \vp(g(b_1)) = g(b_1) \), and hence \( \vp(\Omega_i) \cap \Omega_i \neq \varnothing \).
Therefore, \( \psi(\Omega_i) \cap \Omega_i = \varnothing \) for every \( \psi \in U_1 \), and hence \( U_1 \cap U_2^g = \varnothing \).

Now assume instead that \( g(b_1) \) has nontrivial geometric intersection with \( \Sigma \).
Let \( c \in \cc(\Sigma) \) be a projection of \( g(b_1) \) to \( \Sigma \). 
Any \( \psi \in U_1 \) acts on \( \mathcal{C}(\Sigma) \), and by the definitions of \( f \) and \( U_1 \) we have that \( d_{\cc(\Sigma)}(\psi(c), c) > 2 \).
However, for every \( \vp \in U_2^g \) we have  \( \vp(g(b_1)) = g(b_1) \); this implies \( \vp(c) \) is also a projection of \( g(b_1) \) to \( \Sigma \) and \( d_{\cc(\Sigma)}(\vp(c),c) \leq 2 \).
Therefore we again have \( U_1 \cap U_2^g = \varnothing \), and so conclude that \( \mcg(S) \) fails to have the JEP.
\end{proof}

\subsection{Proof of Theorem~\ref{thm:self-similar}}
\label{sec:proof4}

\begin{proof}[Proof of Theorem~\ref{thm:self-similar}]
By Theorem~\ref{thm:partition}, under the assumptions on \( S \), the set of maximal ends of \( S \) is either a singleton or a Cantor space.
Moreover, as already noted, every surface of positive finite genus has a compact non-displaceable subsurface, so \( S \) has either zero or infinite genus. 
By the equivalence of the JEP and the Rokhlin property for Polish groups (Theorem~\ref{thm:ergodic}) and the fact that mapping class groups are Polish, the theorem readily follows from  Lemma~\ref{lem:singleton-yes} and Lemma~\ref{lem:cantor-no}.
\end{proof}

%---------------
% Doubly pointed end space
%---------------

\section{Doubly pointed end space}
\label{sec:doubly-pointed}

In this section, we present the final case to Theorem~\ref{thm:main}:

\begin{Thm}
\label{thm:doubly-pointed}
The mapping class group of a connected orientable 2-manifold with a doubly pointed end space does not have the Rokhlin property.
\end{Thm}

By Theorem~\ref{thm:nondisplaceable}, we already know that if a 2-manifold contains a compact non-displaceable subsurface, then it cannot have the Rokhlin property; therefore, we focus on the case in which all compact subsurfaces are displaceable.  
Under this assumption, we will show that the pair of maximal ends of a surface with a doubly pointed end space can be interchanged by a homeomorphism of the surface.
This implies that the stabilizer of a maximal end is a proper open normal subgroup, and hence the group cannot have the Rokhlin property.  
This argument will unfold in the sequence of lemmas that follow.

\begin{Lem}
\label{lem:finite-index}
Let \( S \) be a connected orientable surface.
If there exists a compact subsurface \( \Sigma \) of \( S \) and a finite-index subgroup \( H \) of \( \Homeo^+(S) \) such that \( h(\Sigma) \cap \Sigma \neq \varnothing \) for every \( h \in H \), then \( S \) contains a non-displaceable subsurface.
\end{Lem}

\begin{proof}
Let \( n \) denote the index of \( H \) in \( \Homeo^+(S) \) and fix \( f_1, \ldots, f_n \) in \( \Homeo^+(S) \) such that
\[ \Homeo^+(S) = \bigcup_{j=1}^n ( f_j \cdot H) \]
and define 
\[ \Sigma' = \bigcup_{i=1}^n f_j(\Sigma) . \]
If \( g \in \Homeo^+(S) \), then we can write \( g = f_k h \) for some \( k \in \{1, \ldots, n\} \) and \( h \in H \).
Since \( h(\Sigma) \cap \Sigma \neq \varnothing \), it follows that \( (f_k \circ h)(\Sigma') \cap \Sigma' \neq \varnothing \).
Therefore, any compact subsurface containing \( \Sigma' \) is non-displaceable. 
\end{proof}

\begin{Lem}
\label{lem:self-similar-nbhd}
Let \( S \) be a connected orientable 2-manifold with doubly pointed end space \( \sE \) such that every compact subset of \( S \) is displaceable.
Let \( \mu \) be a maximal end of \( S \).
If \( \mathscr W \) is a clopen neighborhood of \( \mu \) in \( \sE \) and \( \mathscr U \) is a clopen subset of \( \sE \) disjoint from the set of maximal ends, then there exists a clopen subset \( \mathscr V \) of \( \sE \) such that \( \mathscr V \subset \mathscr W \) and \( \mathscr V \) is homeomorphic to \( \mathscr U \).
\end{Lem}

Figure~\ref{fig:L53} illustrates the proof of Lemma~\ref{lem:self-similar-nbhd} in the case where \(S\) is a surface with exactly two maximal ends, each accumulated by isolated planar ends as well as isolated non-planar ends.

\begin{figure}[h]
 \labellist
 \small\hair 2pt

  \pinlabel {\(a\)} [ ] at 930 450
    \pinlabel {\(b\)} [ ] at 1060 542
    \pinlabel {\(c\)} [ ] at 1476 450
    \pinlabel {\(\Sigma\)} [ ] at 1350 520
    
    \pinlabel {\(\mu\)} [ ] at 60 140
                \pinlabel {\(\mathscr W\)} [ ] at 720 20
                        \pinlabel {\(\mathscr U\)} [ ] at 990 20
                                               \pinlabel {\(\mathscr V\)} [ ] at 720 240
                        
  \pinlabel {\(\Omega(a)\)} [ ] at 800 520
    \pinlabel {\(\Omega(b)\)} [ ] at 1206 585
    \pinlabel {\(\Omega(c)\)} [ ] at 1602 520                        
                        
  \pinlabel {\(\dots\)} [ ] at 132 540
    \pinlabel {\(\dots\)} [ ] at 1905 540

 \endlabellist  
\centering
\includegraphics[width=\textwidth]{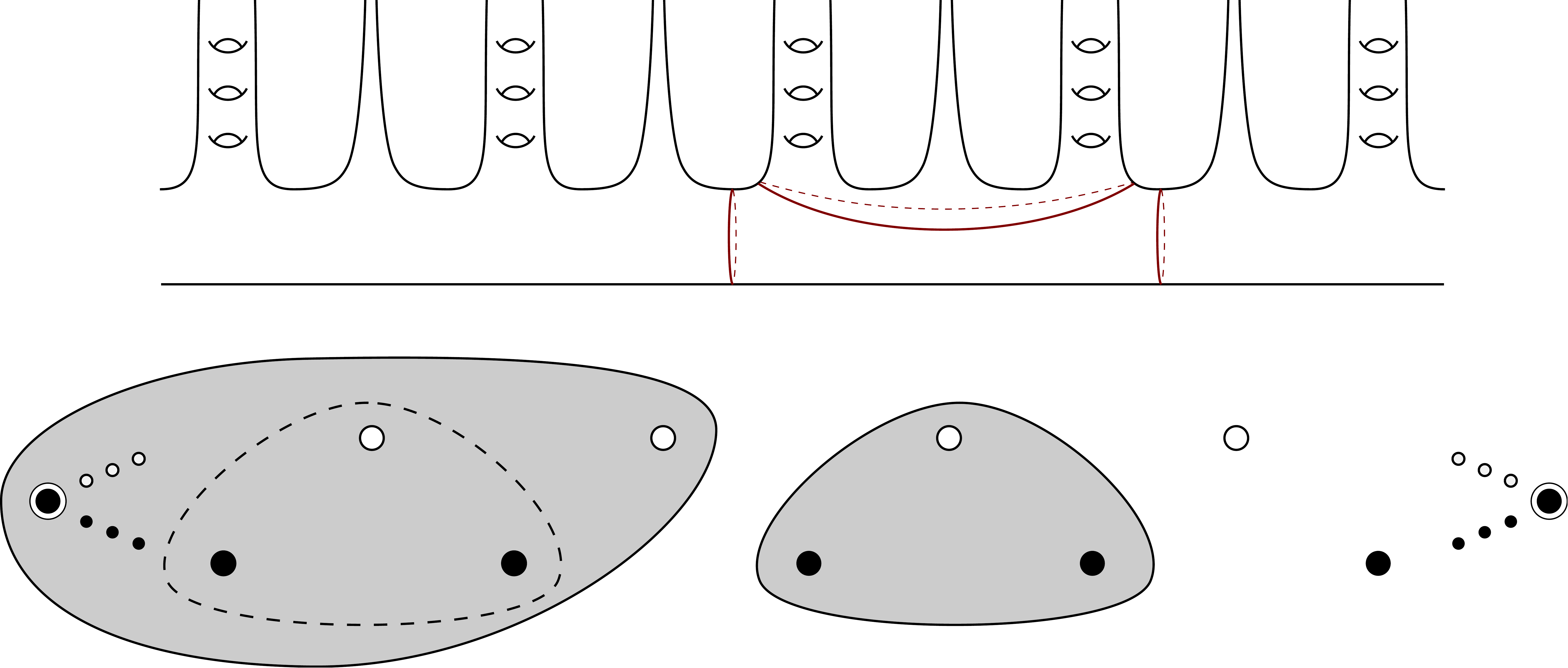}
\caption{The 2-manifold \(S\) (top). The end space \(\sE\) of \( S \) (bottom). }
\label{fig:L53}
\end{figure}
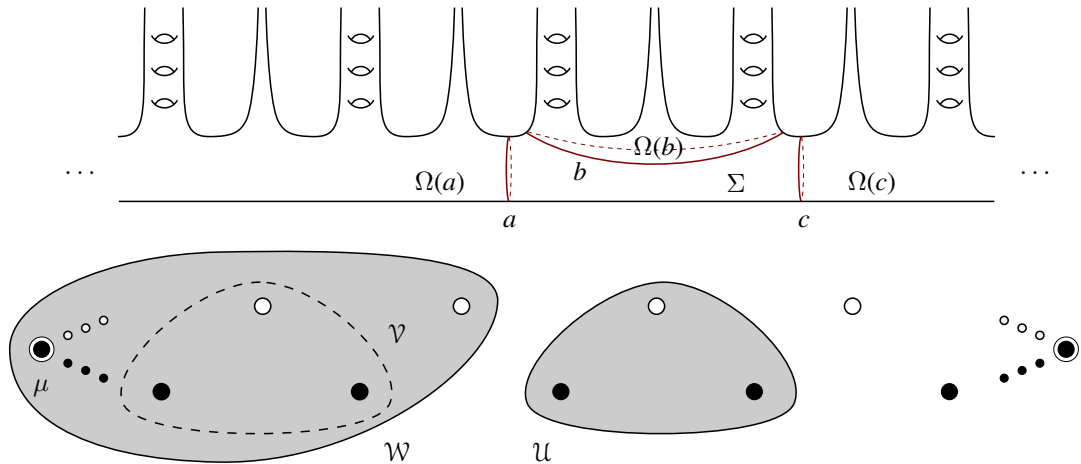

\begin{proof}[Proof of Lemma~\ref{lem:self-similar-nbhd}]
If \( \mathscr U \) is empty, which is necessarily the case if \( S \) is a once-punctured plane, then the statement is trivially true.
We now assume that \( \mathscr U \) is not empty.
By possibly shrinking \( \mathscr W \),  we may assume that \( \mathscr W \cap \mathscr U = \varnothing \).
Let \( H \) denote the stabilizer of \( \mu \) in \( \Homeo^+(S) \) (and hence \( H \) is of index at most 2). 
Let \( a \) and \( b \) be disjoint separating simple closed curves such that there are components \( \Omega_a \) and \( \Omega_b \) of \( S \ssm a \) and \( S \ssm b \), respectively, such that \( \widehat \Omega_a = \mathscr W \), \(  \widehat \Omega_b = \mathscr U \), and \( \Omega_a \cap \Omega_b = \varnothing \). 
Now, let \( \Sigma \) be a pair of pants with \( a \) and \( b \) as two of its boundary components; let \( c \) denote the third component and let \( \Omega_c = S \ssm (\Omega_a \cup \Omega_b \cup \Sigma) \). 
Since \( \Sigma \) is compact and every compact subset is displaceable, by Lemma \ref{lem:finite-index}, there exists \( f \in H \) such that \( f(\Sigma) \cap \Sigma = \varnothing \). 
Note that \( b \) does not separate the maximal ends of \( S \) and that \( c \) does.
Therefore, if \( f(\Sigma) \subset \Omega_a \), then we set \( \mathscr V = \widehat f(\mathscr U) \), and \( \mathscr V \) is necessarily contained in \( \mathscr W \).
Otherwise, \( f(\Sigma) \subset \Omega_c \) and, since \( \widehat f(\mu) = \mu \), it follows that \( \mathscr U \subset \widehat f(\mathscr W) \): we define \( \mathscr V = f^{-1}(\mathscr U) \).
In either case, \( \mathscr V \) is the desired set.
\end{proof}

\begin{Lem}
\label{lem:permuting}
Let \( S \) be a connected orientable 2-manifold with doubly pointed end space \( \sE \) and such that every compact subsurface is displaceable. 
If  \( \mathscr V \) and \( \mathscr W \) are  clopen subsets of \( \sE \) each of which contains a single maximal end of \( \sE \), then \( \mathscr V \) and \( \mathscr W \) are homeomorphic.
In particular, there exists a homeomorphism of \( S \) permuting the two maximal ends. 
\end{Lem}

\begin{proof}
If either \( \mathscr V \) or \( \mathscr W \) is a singleton, which is necessarily the case if \( S \) is a once-punctured plane, then the statement is trivially true.
We now assume that neither set is a singleton. 
We apply a standard back-and-forth argument.
Let \( \mu_1 \)denote the maximal end contained in \( \mathscr V \) and \( \mu_2 \) the maximal end contained in \(  \mathscr W \).
Choose sequences \( \{\mathscr V_n\}_{n\in\bn} \) and \( \{\mathscr W_n\}_{n\in\bn} \) of clopen neighborhoods in \( \sE \) of \( \mu_1 \) and \( \mu_2 \), respectively, such that 
\begin{enumerate}
\item \( \mathscr V_1 = \mathscr V \) and \( \mathscr W_1 = \mathscr W \),
\item \( \mathscr V_{n+1} \subset \mathscr V_{n} \) and \( \mathscr W_{n+1} \subset \mathscr W_{n} \), and
\item \( \bigcap_{n\in\bn} \mathscr V_n = \{\mu_1\} \) and \( \bigcap_{n\in\bn} \mathscr W_n = \{\mu_2\} \).
\end{enumerate}

By Lemma~\ref{lem:self-similar-nbhd}, there exists a homeomorphism \( f_1 \) from \( \mathscr V_1 \ssm \mathscr V_2 \) into \( \mathscr W_1 \ssm \{\mu_2\} \) (we know the image misses \( \mu_2 \) since \( \mathscr V_1 \ssm \mathscr V_2 \) does not contain a maximal end). 
Again by Lemma~\ref{lem:self-similar-nbhd}, there exists a homeomorphism  \( g_1 \) from \( (\mathscr W_1 \ssm \mathscr W_2) \ssm \mathrm{image}(f_1) \) into \( \mathscr V_2 \ssm \{ \mu_1 \} \).
Let \( \mathscr V_1' = (\mathscr V_1 \ssm \mathscr V_2 ) \cup \mathrm{image}(g_1) \) and \( \mathscr W_1' = (\mathscr W_1 \ssm \mathscr W_2) \cup \mathrm{image}(f_1) \), then \[ h_1 = f_1 \sqcup g_1^{-1} \co \mathscr V_1' \to \mathscr W_1' \] is a homeomorphism. 

Recursively, let \( f_n \) be a homeomorphism from \( (\mathscr V_n \ssm \mathscr V_{n+1}) \ssm (\mathscr V_1' \cup \cdots \cup \mathscr V_{n-1}') \) into \( \mathscr W_n \ssm (\mathscr W_1' \cup \cdots \cup \mathscr W_{n-1}') \) and then choose a homeomorphism \( g_n \)  of \( (\mathscr W_n \ssm \mathscr W_{n+1}) \ssm ( \mathscr W_1' \cup \cdots \cup \mathscr W'_{n-1} \cup \mathrm{image}(f_n)) \) into \( \mathscr V_{n+1} \ssm ( \mathscr V_1' \cup \cdots \cup \mathscr V_{n-1}') \) (the maps \( f_n \) and \( g_n \) exist by Lemma~\ref{lem:self-similar-nbhd}).
Let \[ \mathscr V_n' = ((\mathscr V_n \ssm \mathscr V_{n+1}) \ssm (\mathscr V_1' \cup \cdots \cup \mathscr V_{n-1}')) \cup \mathrm{image}(g_n) \] and \[ \mathscr W_n' = ((\mathscr W_n \ssm \mathscr W_{n+1}) \ssm ( \mathscr W_1' \cup \cdots \cup \mathscr W'_{n-1})) \cup \mathrm{image}(f_n). \]
Then, \( h_n = f_n \sqcup g_n^{-1} \co \mathscr V_n' \to \mathscr W_n' \) is a homeomorphism. 
Now observe that \( \mathscr V \ssm \{\mu_1\} = \bigsqcup_{n\in\bn} \mathscr V_n' \) and \( \mathscr W \ssm \{\mu_2\} = \bigsqcup_{n\in\bn} \mathscr W_n'  \).
This allows us to define the desired homeomorphism \( h \co \mathscr V \to \mathscr W \)  by \( h|_{\mathscr V_n'} = h_n \) and \( h(\mu_1) = \mu_2 \).
\end{proof}

\begin{proof}[Proof of Theorem~\ref{thm:doubly-pointed}]
Let \( S \) be a connected orientable 2-manifold with doubly pointed end space.
If \( S \) has a non-displaceable compact subsurface, then \( \mcg(S) \) does not have the Rokhlin property by Theorem~\ref{thm:nondisplaceable}.
Otherwise, by Lemma \ref{lem:permuting}, the stabilizer of a maximal end is a proper normal subgroup of \( \mcg(S) \); moreover, it is not difficult to see that this subgroup is open, and so, again, \( \mcg(S) \) does not have the Rokhlin property.
\end{proof}

\section{Classifying  mapping class groups with the Rokhlin \\ property: the proof of Theorem~\ref{thm:main}}
\label{sec:proof}

The goal of this section is to prove our main theorem:

\begin{Thm}
\label{thm:main}
The mapping class group of a connected orientable 2-manifold has the Rokhlin property if and only if the manifold is either the 2-sphere or a non-compact manifold whose genus is either zero or infinite and whose end space is self-similar with a unique maximal end.\end{Thm}

\begin{proof}%[Proof of Theorem~\ref{thm:main}]
First, we note that the mapping class group of the 2-sphere is trivial and hence trivially has a dense conjugacy class.
Now, let \( S \) be a connected orientable 2-manifold that is non-compact, whose end space is self-similar with a unique maximal end, and whose genus is either zero or infinite.
It readily follows from Proposition~\ref{prop:top-generate}(2) that every compact subsurface of \( S \) is displaceable, and hence, by Theorem~\ref{thm:self-similar}, \( \mcg(S) \) has the Rokhlin property.

For the converse, by Theorem~\ref{thm:partition}, every connected orientable 2-manifold satisfies the hypotheses of at least one of Theorem~\ref{thm:nondisplaceable}, Theorem~\ref{thm:self-similar}, or Theorem~\ref{thm:doubly-pointed}.
In which case, we see that if the mapping class group of such a 2-manifold has the Rokhlin property, then the 2-manifold must either be a 2-sphere or a non-compact manifold whose end space is self-similar with a unique maximal point and in which every compact subsurface is displaceable.
The last condition forces the surface to have either zero or infinite genus, as desired. 
\end{proof}

\begin{Rem}
%It is unclear whether the condition of self-similarity in the statement of Theorem~\ref{thm:main} is necessary. 
%This is tantamount to asking whether there exists an end space of a surface that has a unique maximal end that fails to be self-similar.
%The authors expect no such space to exist. 
An earlier version of this note claimed that the self-similarity condition was unnecessary in the statement of Theorem~\ref{thm:main}; however,  the proof contained an error, which was pointed out by Malestein and Tao.
In response to this issue, Mann and Rafi \cite{MannNon} have since shown that there exists a 2-manifold whose end space has a unique maximal end and yet fails to be self-similar. 
This shows that the self-similarity assumption in Theorem~\ref{thm:main} is in fact necessary.\end{Rem}

%-----------------
% No comeager conugacy class
%-----------------

\section{No generic mapping classes}
\label{sec:meager}

Before stating the main theorem in this subsection, we need to recall some basic topological definitions. 
A subset \( A \) of a topological space is \emph{nowhere dense} if the interior of its closure is empty; it is \emph{meager} if it is the countable union of nowhere dense subsets; it is \emph{comeager} if its complement is meager; and, it has the \emph{Baire property} if there exists an open subset \( U \) such that the symmetric difference \( A\triangle U \) is meager.

Following Truss \cite{TrussGeneric}, we say an element \( f \) of a Polish group \( G \) is \emph{generic} if its conjugacy class is comeager. Our final result shows that no nontrivial mapping class group has a generic element.

\begin{Thm}
\label{thm:meager}
The mapping class group of an orientable 2-manifold has a comeager conjugacy class if and only if the mapping class group is trivial. 
\end{Thm}

Before proving Theorem~\ref{thm:meager}, we need some more preliminaries, both from surface topology and from the theory of Polish groups.  We start with the latter.
The statement and argument of the following proposition is a modified version of the one given by Tent--Ziegler in their online notes \cite{TentZiegler} on an article of Kechris--Rosendal \cite{KechrisRosendal}.
In what follows, \( f^G \) will denote the conjugacy class of an element \( f \) in a group \( G \).
We say that a set \( A \) is \emph{comeager on an open set} \( U \) if \( A \cap U \) is comeager in \( U \). 

\begin{Prop}
\label{prop:weakly-generic}
Let \( G \) be a Polish group, let \( f \in G \) have a non-meager conjugacy class, and let \( U \subset G \) be any open neighborhood of \( f \). 
If \( H \) is a subgroup of the stabilizer of \( U \) such that \( H \) is \( G_\delta \) and has countable index in \( G \), then there exists a open neighborhood \( V \) of \( f \) contained in \( U \) such that \( f^H \) is comeager on \( V \).
\end{Prop}

\begin{proof}
Since \( H \) is countable index in \( G \),  \( f^G \) is a countable union of translates of \( f^H \) under the conjugation action of \( G \).
Therefore, if \( f^H \) were meager, then \( f^G \) would be a countable union of meager sets and hence meager itself.
But, this would contradict \( f^G \) being non-meager.
Therefore, \( f^H \) is non-meager.

It is a standard fact that \( G_\delta \) subsets of Polish spaces are Polish (see \cite[Theorem~3.11]{KechrisDescriptive}), and hence \( H \) is Polish. 
Moreover, \( f^H \) is the image of \( H \) under the continuous map \( H \to G \) given by \( h \mapsto hfh^{-1} \), and hence, by a classical result of Lusin--Sierpi\'nski (see~\cite[Theorem 21.6]{KechrisDescriptive}), \( f^H \) has the Baire property.

This guarantees the existence of an open subset \( W \) of \( G \) such that \( f^H \triangle W \) is meager.
Without loss of generality, we may assume \( f \in W \): for if not, then since \( W \ssm f^H \) is meager and \( W \) is non-empty,  \( f^H \cap W \) is non-empty.
Let \( h^{-1}fh \in f^H \cap W \) for some \( h \in H \), then \( f \in (f^H \cap W)^h = f^H \cap W^h \).
It follows that \( (f^H \triangle W)^h = f^H \triangle W^h \) is meager, and so we can replace \( W \) with \( W^h \).
Note that \( f^H \) is comeager on every open subset of \( W \); hence, we simply let \( V \) be any open neighborhood of \( f \) contained in \( U \cap W \). 
\end{proof}

Now, we must discuss surface topology.
Let \( \Sigma \) denote a surface with non-empty boundary. 
A \emph{proper} arc in \( \Sigma \) is the image of a continuous mapping \( \alpha \co [0,1] \to \Sigma \) such that \( \alpha(0), \alpha(1) \in \partial \Sigma \), and \( \alpha|_{(0,1)} \) is a proper map from \( (0,1) \) to \( \Sigma \ssm \partial \Sigma \). 
A proper simple arc \( \alpha \) is \emph{essential} if no component of \( \Sigma \ssm (\alpha \cup \partial \Sigma) \) is homeomorphic to a disk. 
We let \( \cac(\Sigma) \) denote the set of isotopy classes\footnote{This definition is slightly non-standard in the sense that it is generally required that the isotopies are taken relative to \( \partial \Sigma \); however, this point will not affect our usage.} of arcs on \( \Sigma \) together with the isotopy classes of essential simple closed curves on \( \Sigma \).
As with \( \cc(\Sigma) \), we say two elements of \( \cac(\Sigma) \) are adjacent if they have disjoint representatives, and we let \( d_{\cac(\Sigma)} \) denote the associated graph metric. 
Now, let \( \Sigma \) be a subsurface of a 2-manifold \( S \), then for any \( c \in \cc(S) \) that has nontrivial geometric intersection with \( \partial \Sigma \), each component of \( c \cap \Sigma \) yields an element of \( \cac(\Sigma) \) (assuming \( c \) and \( \partial \Sigma \) are in minimal position); moreover, the set of all such components has diameter two just in \( \cac(\Sigma) \). 
Standard properties of the embedding of \( \cc(\Sigma) \) into \( \cac(\Sigma) \) allow us to use results from \cite{MasurGeometryI}; in particular, if \( \Sigma \) is finite type and \( \cc(\Sigma) \) is connected, then for any \( D > 0 \) there exists \( g \in \mcg(\Sigma) \) such that \( d_{\cac(\Sigma)}(c,g(c))> D \) for all \( c \in \cac(\Sigma) \).
For a reference, see \cite[Section 3]{SchleimerNotes}.

Figure~\ref{fig:L53} gives a schematic for the curves and subsurfaces used in the proof of Theorem~\ref{thm:meager}.

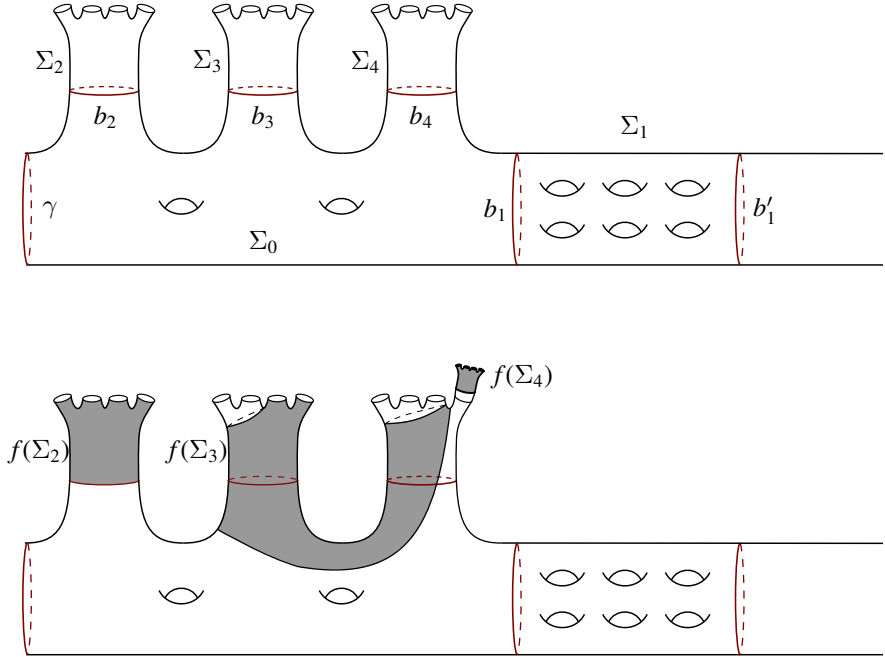
\begin{figure}[t]
 \labellist
 \small\hair 2pt

 \pinlabel {$\gamma$} [ ] at 30 480
  \pinlabel {$b_1$} [ ] at 510 480
    \pinlabel {$b_1'$} [ ] at 800 480

    \pinlabel {$\Sigma_0$} [ ] at 260 445    
    \pinlabel {$\Sigma_1$} [ ] at 660 570
    
        \pinlabel {$\Sigma_2$} [ ] at 30 640
        \pinlabel {$\Sigma_3$} [ ] at 200 640
        \pinlabel {$\Sigma_4$} [ ] at 370 640

   \pinlabel {$b_2$} [ ] at 90 580
     \pinlabel {$b_3$} [ ] at 260 580
       \pinlabel {$b_4$} [ ] at 430 580
       
               \pinlabel {$f(\Sigma_2)$} [ ] at 20 220
        \pinlabel {$f(\Sigma_3)$} [ ] at 190 220
        \pinlabel {$f(\Sigma_4)$} [ ] at 540 300
 \endlabellist  
\centering
\includegraphics[scale=.35]{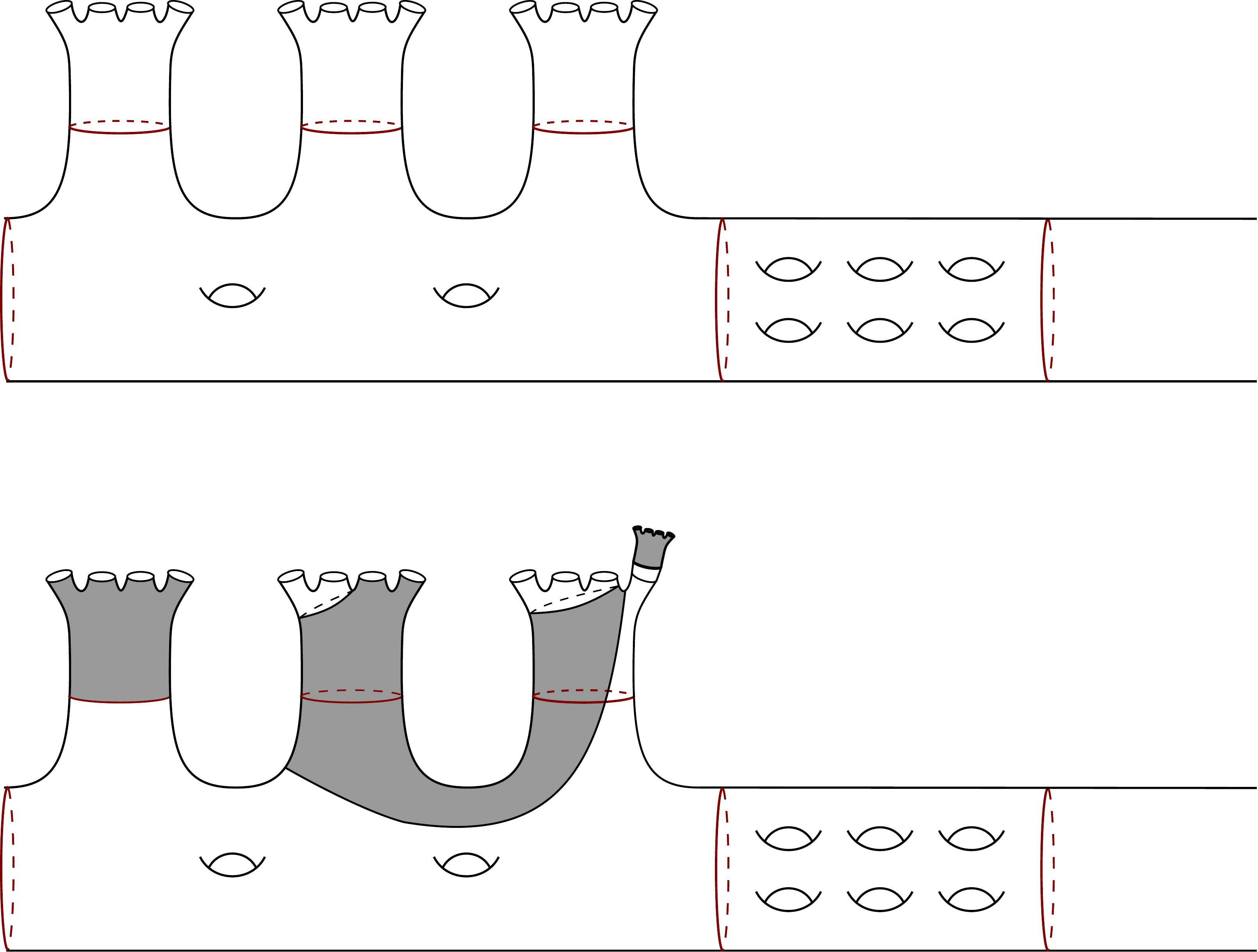}
\caption{The curves and subsurfaces used in the proof of Theorem~\ref{thm:meager} (top). The three scenarios arising in the definition of the elements \( g_j \in \mcg(\Sigma_j) \) for  \( j \in \{2, \ldots, n\} \) (bottom). }
\label{fig:T71}
\end{figure}

\begin{proof}[Proof of Theorem~\ref{thm:meager}]
Let \( S \) be a non-simply connected orientable 2-manifold.
Assume that \( \mcg(S) \) has an element \( f \) with a comeager conjugacy class.
In a Polish space, comeager sets are dense, and hence the conjugacy class of \( f \) is dense.
Therefore, by Theorem~\ref{thm:main}, \( S \) has either zero or infinite genus and a self-similar end space with a unique maximal end.
Let \( \gamma \in \cc(S) \) be separating and let \( U = U_{\{\gamma\}} \).
Then, by the density of the conjugacy class of \( f \), there exists a conjugate of \( f \) in \( U \):
by replacing \( f \) with this conjugate, we may assume without loss of generality that \( f \in U \).

Let \( H \) denote the stabilizer of \( \gamma \), and hence of \( U \), in \( \mcg(S) \), and note that \( H \) is closed (and hence \( G_\delta \)) and countable index in \( \mcg(S) \) (in fact, \( H = U\)).  
Let \( V \) be the neighborhood of \( f \) contained in \( U \) guaranteed by Propsoition~\ref{prop:weakly-generic}; in particular, \( f^H \cap V \) is comeager, and hence dense, in \( V \).
By Proposition~\ref{prop:top-generate}, there exists a separating simple closed  curve \( b_1 \)  and a mapping class \( \psi \in V \) such that \( \psi \) has a representative that restricts to the identity on the complementary component of \( b_1 \) that is a neighborhood of the unique maximal end. 
For convenience, we may choose \( b_1 \)  to be homotopically distinct from \( \gamma \).

Now, choose a connected finite-type subsurface \( \Sigma_0 \) with stable Alexander system \( A_0 \) satisfying
\begin{itemize}
\item \( \Sigma_0 \) is a closed subset of \( S \),
\item \( fU_{A_0} \subset V \),
\item \( \gamma \) and \( b_1 \) are components of \( \partial \Sigma_0 \), 
\item each component of \( \partial \Sigma_0 \)  is separating,
\item each complementary component of \( \Sigma_0 \)  is of infinite type.
\end{itemize}
Note that \( S \) has negative Euler characteristic and hence every subsurface that has at least two essential non-homotopic boundary components also has negative Euler characteristic,  such as \( \Sigma_0 \). 
Let \( \gamma, b_1, b_2, \ldots, b_n \) be a labelling of the boundary components of \( \partial \Sigma_0 \).
Let \( \Sigma_1 \) be a connected finite-type surface of \( S \) satisfying
\begin{itemize}
\item \( \partial \Sigma_1 \) has two boundary components,  \( b_1 \) and another curve we label \( b_1' \),
\item \( \Sigma_1 \cap \Sigma_0 = b_1 \),
\item \( \Sigma_1 \) either has positive genus or \( \widehat \Sigma_1 \) contains at least three ends, and 
\item the Euler characteristic of \( \Sigma_1 \) is strictly less than that of \( \Sigma_0 \) . 
\end{itemize}
For each \( j \in \{2, \ldots, n\} \), choose a connected compact subsurface \( \Sigma_j \) whose intersection with \( \Sigma_0 \) is \( b_j \) and such that \( \Sigma_j \) is homeomorphic to the 2-sphere with five pairwise-disjoint open disks removed. 
Let \( \Sigma = \bigcup_{j=0}^n \Sigma_j \) and let \( A \) be a stable Alexander system for \( \Sigma \), and note that \( fU_A \subset fU_{A_0}\subset V \). 

Let \( A_1 \) be a stable Alexander system for \( \Sigma_1 \).
Then, \( \psi\in U_{A_1} \) and hence \( V \cap U_{A_1} \) is nonempty.
In particular, there must exist an \( H \)-conjugate of \( f \) contained in this intersection, and hence, without loss of generality, we may assume that \( f \in U_{A_1} \). 

It will be helpful to fix a representative of \( f \), which we again denote by \( f \), satisfying the following:
For any component \( \delta \) of \( \partial \Sigma_j \) and any component \( \delta' \) of \( \partial \Sigma_i \),   \( \delta = f(\delta') \) and \( f \) fixes \( \delta \) pointwise whenever \( \delta \) and \( f(\delta') \) are homotopic (such as when \( \delta = b_1 \)), and \( | \delta \cap f(\delta') | = i(\delta,f(\delta')) \) whenever \( \delta \) and \( f(\delta') \) are not homotopic. 
%
%\( | \partial \Sigma_i \cap \partial f(\Sigma_j)| = i(\partial \Sigma_i, \partial f(\Sigma_j)) \) whenever 
%
%\( f \) fixes a component of \( \partial \Sigma_j \) pointwise whenever 
%
%when a component \( \partial \Sigma_j \) is homotopic to a component of \( \partial f(\Sigma_j) \),  \( f \) fixes \( \partial \Sigma_j \) pointwise, and when a component of \( \partial \Sigma_i \) and a component of \( \partial f(\Sigma_j) \) are non-homotopic,  . 

We let \( W_1 = fU_A \).
We now build an open subset  \( W_2=gfU_A \) of \( V \) by constructing a mapping class \( g \).
We will go on to show that every \( H \)-conjugate of \( W_2 \) is disjoint from \( W_1 \), which will contradict the assumption that the conjugacy class of \( f \) is comeager. 
By construction, the embedding of \( f(\Sigma_j) \) into \( S \) induces a monomorphism of \( \mcg(f(\Sigma_j)) \) into \( \mcg(S) \); in what follows, we will identify \( \mcg(f(\Sigma_j)) \) with its image in \( \mcg(S) \) under this monomorphism. 
Let \[ K = \sum_{2\leq i,j \leq n} i(b_i, f(b_j)). \]
Choose \( g_1 \in \mcg(\Sigma_1) \) such that \( d_{\cc(\Sigma_1)}(c,g_1(c)) > 2K+9 \) for every \( c \in \mathcal{C}(\Sigma_1) \). 
Next, we need to carefully choose \( g_j \in \mcg(\Sigma_j) \) for  \( j \in \{2, \ldots, n\} \): there are three scenarios. 

(i) If \( f(\Sigma_j) \) and \( \Sigma_j \) are homotopic, then we may write \( f = f_j'\circ f' \), where \( f_j' \) has a representative that restricts to the identity on the complement of \( b_j \) containing \( \Sigma_0 \) and \( f' \) has a representative that restricts to the identity on the complement of \( b_j \) containing \( \Sigma_j \).
In this case, we choose  \( g_j' \in\mcg(\Sigma_j) \) such that \( d_{\cac(\Sigma_j)}(g_j'(c), c) > 1 \) for all \( c \in \cac(\Sigma_j) \), and let \( g_j = g_j'\circ (f_j')^{-1} \).

(ii) If \( f(\Sigma_j) \) and \( \Sigma_j \) have nontrivial geometric intersection, but are not homotopic, then let \( \mathcal A_j \) denote all the simple arcs in the intersection of \( \Sigma_j \) and \( f(\Sigma_j) \) with endpoints on \( \partial \Sigma_j \cap \partial f(\Sigma_j) \).  
Let \( \beta \) be a component of the intersection of \( \partial \Sigma_j \) with \( f(\Sigma_j) \), then \( \beta \in \cac(f(\Sigma_j)) \) and \( d_{\cac(f(\Sigma_j))}(\beta, \alpha) = 1 \) for all \( \alpha \in \mathcal A_j \); hence, the diameter of \( \mathcal A_j \) in \( \cac(f(\Sigma_j)) \) and the diameter of \( f(\mathcal A_j) \) in \( \cac(f(\Sigma_j)) \) are both equal to two. 
Therefore, the diameter \( d_j \) of \( \mathcal A_j \cup f(\mathcal A_j) \) is finite in \( \cac(f(\Sigma_j)) \).
Let \( g_j \in \mcg(f(\Sigma_j)) \) such that \( d_{\cac{(f(\Sigma_j))}}(c, g_j(c)) > d_j \) for all \( c \in \cac(f(\Sigma_j)) \).

(iii) If \( f(\Sigma_j) \) does not fit into the above two cases, then simply let \( g_j \) be the identity (this case will not play a role in the proof).

Set \[ g = \prod_{j=1}^n g_j \]
and we have \( W_2 = gfU_A \).
Observe that \( W_1, W_2 \subset V \).
Therefore, by the assumption that the conjugacy class of \( f \) is comeager, there exists an \( H \)-conjugate of \( f \) in \( W_2 \), which allows us to conclude that there exists \( h \in H \) such that \( W_1\cap W_2^h \neq \varnothing \).  
The goal is to show that such an \( h \) cannot not exist: suppose to the contrary that \( h \) exists and let \( \vp \in W_1 \cap W_2^h \).
The remainder of the proof splits into three cases.

\textbf{Case 1:} \( h(\Sigma_1) \) and \( \Sigma_1 \) have nontrivial geometric intersection.
First note that it cannot be that \( h(\Sigma_1) \) is homotopic to \( \Sigma_1 \), since then \( \vp(a) = a \) and  \( \vp(a) = g_1(a) \) for each \( a \in A_1 \), which is of course impossible.
Therefore, there exists a curve \( a \in \cc(\Sigma_1) \) with a nontrivial projection \( b \) to \( \cc(h(\Sigma_1)) \). 
Then, \( d_{\cc(\Sigma_1)}(g_1h^{-1}(b), h^{-1}(b)) > 2 \), and so \[ d_{\cc(h(\Sigma_1))}(\vp(b), b) = d_{\cc(h(\Sigma_1))}(hg_1h^{-1}(b), b)>2 .\] 
But since \( \vp(a) = f(a) = a \), it must be that \( \vp(b) \) is a projection of \( a \) to \( h(\Sigma_1) \), and hence \( d_{\cc(h(\Sigma_1))}(\vp(b), b) \leq 2 \), contradicting the above inequality.

\textbf{Case 2:} \( h(\Sigma_1) \) and \( \Sigma_1 \) have trivial geometric intersection, but \( h(\Sigma_1) \) and \( \Sigma_0  \) have nontrivial geometric intersection. 
First observe that the restriction on the Euler characteristic on \( \Sigma_1 \) guarantees that \( h(\Sigma_1) \) is not contained in \( \Sigma_0 \). 
This guarantees that there exists \( j \in \{2, \ldots, n\} \) such that either \( i(b_j, h(b_1)) \neq 0 \) or \( i(b_j, h(b_1')) \neq 0 \);  assume the former is true (the argument is the same for the other case). 
It follows that  \( b_j \) has a nontrivial projection \( b \) to \( \cc(h(\Sigma_1)) \).

\begin{align*}
d_{\cc(h(\Sigma_1))}(\vp(b), b) 	&= d_{\cc(h(\Sigma_1))}(hgfh^{-1}(b), b) \\
						&= d_{\cc(\Sigma_1)}(gfh^{-1}(b), h^{-1} (b))\\
						&= d_{\cc(\Sigma_1)}(g_1(h^{-1}(b)),h^{-1}(b)) \\
						&> 2K+9.
\end{align*}

However, \( \vp(b_j) = f(b_j), \vp(h(b_1))=h(b_1) \), and \( \vp(h(b_1'))=h(b_1') \); hence, \( \vp(b) \) must be a projection of \( f(b_j) \) to \( h(\Sigma_1) \).
Since \( i(b_j, f(b_j)) \leq K \), we have that \( i(\vp(b), b) \leq K+4 \) and hence \[ d_{\cc(h(\Sigma_1))}(\vp(b),b) \leq 2(K+4)+1 = 2K+9, \] where the inequality comes from a standard argument in surface topology, see \cite[Lemma 2.1]{MasurGeometryI}.
Hence, we have arrived at a contradiction.

\textbf{Case 3:} \( h(\Sigma_1) \) has trivial geometric intersection with both \( \Sigma_0 \) and  \( \Sigma_1 \).
Since \( \Sigma_1 \) and all its \( H \)-conjugates separate \( \gamma \) from the maximal end of \( S \), we must have that \( \Sigma_1 \) and \( \gamma \) are on the same side of \( h(\Sigma_1) \). 
And since \( \gamma \) is a boundary component of both \( \Sigma_0 \) and \( h(\Sigma_0) \), we must have that \( h(\Sigma_0) \cap \Sigma_1 \neq \varnothing \).
Again, the restriction on the Euler characteristic of \( \Sigma_1 \) guarantees that \( \Sigma_1 \) is not contained in \( h(\Sigma_0) \). 
In particular, there exists \( j \in \{2, \ldots, n\} \) such that \( i(h(b_j), b_1) \neq 0 \).
Let \( \alpha \) denote a component of \( b_1 \cap h(\Sigma_j) \), so that \( \alpha \subset \Sigma \cap h(\Sigma) \); hence \( \vp(\alpha) = \alpha \) and \( \vp(\alpha) = hgfh^{-1}(\alpha) \).
In particular, recalling the definition of \( g \), we have
\begin{align*}
hg_jfh^{-1}(\alpha)
				&= hgfh^{-1}(\alpha) \\
				&= \vp(\alpha) \\
				&= \alpha,
\end{align*}
and so \( g_jfh^{-1}(\alpha) = h^{-1}(\alpha) \).

On the other hand, observe that \( hgfh^{-1}(h(\Sigma_j)) = h(f(\Sigma_j)) \) since by construction \( g \) maps  \( f(\Sigma_j) \) onto itself; hence, \( \alpha \) is also contained in \( h(f(\Sigma_j)) \). 
In particular, \( \Sigma_j \cap f(\Sigma_j) \) contains an arc, namely \( h^{-1}(\alpha) \), with both endpoints in \( \partial f(\Sigma_j) \cap \partial\Sigma_j \). 
Now, either \( \Sigma_j = f(\Sigma_j) \) or not (corresponding to scenarios (i) and (ii) above, respectively), and in the latter case \( h^{-1}(\alpha) \) is in the set \( \mathcal A_j \).
In either case, by our choice of \( g_j \), we have
\[ d_{\cac(f(\Sigma_j))}(g_jf(h^{-1}(\alpha)),h^{-1}(\alpha)) > 1, \] which contradicts the fact that \( g_jf(h^{-1}(\alpha)) = h^{-1}(\alpha) \). 

These three cases cover all the possibilities, and each results in a contradiction.
Therefore, we can conclude that \( \mcg(S) \) has no comeager conjugacy class. 
\end{proof}

\bibliographystyle{amsplain}
\bibliography{Bib-jep}

\end{document}